\documentclass[11pt, reqno]{amsart}
\usepackage{amssymb,amsmath}
\usepackage{amsthm}
\usepackage[bbgreekl]{mathbbol}
\usepackage{amscd}
\usepackage{stmaryrd}
  \usepackage{hyperref}
    \usepackage{color}
    
    \usepackage{mathabx}
     \usepackage{geometry }
\theoremstyle{plain}

\usepackage{slashed}
\newcommand{\fsl}[1]{{\slashed{#1}}} 

  \DeclareMathOperator{\dd}{\boldsymbol d}
\DeclareMathOperator{\cdd}{\boldsymbol \delta}

\newtheorem{theorem}{Theorem}[section]
\newtheorem{theorem*}[theorem]{Theorem }
\newtheorem{proposition}[theorem]{Proposition}

\newtheorem{lemma}[theorem]{Lemma}
\newtheorem{remark}[theorem]{Remark}

 \newtheorem{example}[theorem]{Example}
 
 \DeclareMathOperator{\Ind}{Ind}
 
\DeclareSymbolFontAlphabet{\mathbb}{AMSb}
\DeclareSymbolFontAlphabet{\mathbbl}{bbold}

\DeclareMathOperator{\id}{id}
\DeclareMathOperator{\Id}{Id}

\DeclareMathOperator{\Str}{Str}
\DeclareMathOperator{\Co}{Co}
\DeclareMathOperator{\GL}{GL}

\DeclareMathOperator{\res}{res}

 \DeclareMathOperator{\Rel}{Re}
 
\numberwithin{equation}{section}

\begin{document}

\title{The source operator method: an overview}
\author{Salem Ben Sa\"id }
\address{ Department of Mathematical Sciences, College of Science, United Arab Emirates University, Al Ain Abu Dhabi, UAE.}
\email{salem.bensaid@uaeu.ac.ae}
\author{Jean-Louis Clerc }
\address{Université de Lorraine, CNRS, IECL, F-54000 Nancy, France}
\email{jean-louis.clerc.Koufany@univ-lorraine.fr}
\author{Khalid Koufany}
\address{Université de Lorraine, CNRS, IECL, F-54000 Nancy, France}
\email{khalid.koufany@univ-lorraine.fr}

 \dedicatory{To Toshiyuki Kobayashi for his commitment to the mathematical community }
\maketitle

 \begin{abstract}  
This is an overview on the {source operator method} which leads to the construction of   symmetry breaking differential operators  (SBDO) in the context of   tensor product of two principals series representations for the conformal group of a simple real Jordan algebra. This method can be applied to other geometric contexts: in the construction of SBDO for  differential forms and  for  spinors,  and also for the construction of Juhl's  operators corresponding to the restriction from the sphere $S^n$ to $S^{n-1}$.
 
   \end{abstract}

   \tableofcontents

   \section{Introduction}

T. Kobayashi's article \cite{kob} entitled
\smallskip

\centerline{\sl  A program for branching problems in the representation theory of reductive group}
\smallskip
\noindent 
 is concerned with the following general  problem : for a real reductive group $G$ and a reductive subgroup $G'$, 
how does an irreducible representation $\pi$ of $G$ behaves when restricted to $G'$ ? His program consists in three stages:
\begin{itemize}
\item  Stage A. Abstract features of the restriction $\pi_{\vert G'}.$
\item Stage B. Branching laws.
\item  Stage C. Construction of symmetry breaking operators.
\end{itemize}

Given a representation $\rho$ of $G'$, a \emph{symmetry breaking operator} $D$ is an operator from the representation space of $\pi$ into the representation space of $\rho$ which satisfies the intertwining property, valid for any $g'\in G'$
\begin{equation}D\circ \pi(g')= \rho(g') \circ D.
\end{equation}

For the history of the subject, examples of symmetry breaking operators came earlier  than Kobayashi's program, and they  gave inspiration for the first  two stages. An early example is the \emph{Rankin-Cohen brackets}  \cite{co}. In this case, the group $G$ is the product ${\rm SL}(2,\mathbb R)\times {\rm SL}(2,\mathbb R)$ and $G'$ is the diagonal subgroup of $G$ naturally identified with ${\rm SL}(2,\mathbb R)$. Let $\mathcal O(\Pi)$ be the space of holomorphic functions on the upper half-pane  $ \Pi = \{ z=x+iy: y>0\}$, and for any $k\in \mathbb N$, let $\pi_k$ be the representation of $SL(2,\mathbb R)$ on $\mathcal O(\Pi)$ given as follows : for $g\in {\rm SL}(2,\mathbb R)$ such that $g^{-1} = \begin{pmatrix} a&b\\c&d\end{pmatrix}$  and $F\in \mathcal O(\Pi)$, set 
\begin{equation}
\pi_k(g) F(z) = (cz+d)^{-k} F\left(\frac{az+b}{cz+d}\right).
\end{equation}
This formula defines a representation $\pi_k$ of ${\rm SL}(2,\mathbb R)$ on $\mathcal O(\Pi)$.  Let $\ell,m$ be two nonnegative integers. The tensor product (in fact an appropriate completion) $\rho_\ell\otimes \rho_m$ is a representation of $G$, which has a natural realization on $\mathcal O(\Pi\times \Pi)$. Let further $k\in \mathbb N$. Consider the following constant coefficients bi-differential operator
\[B_{\ell, m}^{(k)} : \mathcal O(\Pi\times \Pi) \longrightarrow \mathcal O(\Pi)
\]
defined by
\begin{equation}\label{RC-intro}
B_{\ell,m}^{(k)} F(z) = \sum_{r+s=k} (-1)^r \left(\begin{matrix} \ell+k-1\\s \end{matrix}\right)\left(\begin{matrix} m+k-1\\r \end{matrix}\right)\left(\frac{\partial}{\partial z} \right)^r \left(\frac{\partial}{\partial w} \right)^s F (z,w)_{\left| z=w\right.}.
\end{equation}
Then, for any $g\in {\rm SL}(2,\mathbb R)$, we have
\begin{equation} 
B_{\ell,m}^{(k)}\circ \left(\rho_\ell(g)\otimes \rho_m(g) \right)= \rho_{\ell+m+2k}(g) \circ B_{\ell,m}^{(k)}\ ,
\end{equation}
or otherwise stated, $B_{\ell,m}^{(k)}$ is a symmetry breaking operator for the  pair $(\rho_\ell\otimes \rho_m, \rho_{\ell+m+2k})$. See \cite{c23} for more information and proofs on this case.

This article presents a method, called the \emph{source operator method}   to construct such symmetry breaking differential operators (SBDO).  It turned out to be sucessfull in many situations. 

We concentrate on one example (first exposed in \cite{bck1}), which covers a large example of geometric situations and for which a rather explicit formula (see \cite{Clerc-Rod}) for the symmetry breaking operators is now known. In Section 2  a short presentation of the geometric situation was given.  An overview of the construction of the source operator is given in Section 4. This comes  after a functional equation which has been established in Section 3, which  is an important result on its own. The corresponding SBDO are constructed in Section 5. Finally, Section 6 collects other examples where the source operator method is used.

It is not clear at the moment to which cases the source operator method may be applied, both in terms of geometric situations and in terms of representations. Notice first that we were only interested in cases where the source operator is a \emph{differential} operator, but the first step of the construction (i.e. the use of a multiplication operator $M$ and the Knapp-Stein operators) is quite general. It would be interesting to describe all cases where this construction actually leads to a differential operator. In the case of a tensor product representation of the conformal group ${\rm O}(1,n+1)$, the answer is positive for the case of differential forms and also for the spinor case, but the question remains  open for other vector-valued representations. Another question is to characterize the source operator among all possible $G'$ covariant operators. In the tensor product case, it is not difficult to see that there are many ${\rm diag}(G\times G)$-covariant differential operators between $\mathcal H_\lambda\otimes \mathcal H_\mu$ and $\mathcal H_{\lambda+1}\otimes \mathcal H_{\mu+1}$, and our guess is that the one we have constructed is the one which has the lowest possible degree.

Finally, let us mention two connections of the source operator method with other questions. The source operator corresponding to the (Euclidean) conformal case already appeared in \cite {bc} in connection with the study of {singular invariant trilinear forms} on the sphere $S=S^n$. It is connected to a \emph{Bernstein-Sato identity} for the generic kernel of the trilinear form viewed as a distribution on $S\times S\times S$. In a quite different direction, operators similar to the source operator   appeared in \cite{c17}, in connection with conformal analysis on hypersurfaces of manifolds, under the name of \emph{shift operators} (see \cite{jo}) or \emph{degenerate Laplacians} (see \cite{gw}).


\section{Background: Real simple Jordan algebra}
Let $V$ be a  real Jordan algebra of dimension $n$, with a bilinear multiplication $V\times V\rightarrow V, (x,y)\mapsto x\cdot y,$ and a unit element $\bf 1.$ 
To every $x\in V$, we associate a generic minimal polynomial  (see \cite[Section II.2]{FK}) 
$$
f_x(\lambda)=\lambda^r-a_1(x) \lambda^{r-1}+\cdots+(-1)^r a_r(x) .
$$
Its degree $r$ is called the so-called rank of the Jordan algebra $V.$  For $1 \leq j \leq r$, the coefficient $a_j(x)$ are homogeneous polynomials
 of degree $j$ and invariant under automorphisms of $V$. 
In particular, the {Jordan trace} $\operatorname{tr}(x):=a_1(x)$ and the {Jordan determinant} ${\bf det}(x):=a_r(x)$   are invariant under automorphisms of $V$.  

The symmetric bilinear form
\begin{equation} \label{1}  \tau(x, y):=\operatorname{tr}(x y),\qquad x, y \in V\end{equation}
is called the  trace form of $V$. 
The  real Jordan algebra $V$ is said to be {semi-simple} if $\tau$ is non-degenerate, and 
$V$ is called  Euclidean if  $\tau$ is positive definite. Moreover, a  semi-simple Jordan algebra with no nontrivial ideal is called simple.  
From now on we will 	assume that $V$ is a simple real Jordan algebra.
 
An involutive automorphism $\alpha$ of $V$ such that
\begin{equation}\label{ds}
(x \mid y):=\tau(x, \alpha y)
\end{equation}
is positive definite, is called {Cartan involution} of $V$. 
For a semi-simple Jordan algebra, such a Cartan involution always exists and two Cartan involutions are conjugate by an automorphism of $V$ 
(see \cite[Satz 4.1, Satz 5.2]{Helwig}). We have the decomposition
$$
V=V^{+} \oplus V^{-}
$$
into $\pm 1$ eigenspaces of $V$. One can  see that
$$
\begin{aligned}
& V^{\pm} \cdot V^{\pm} \subseteq V^{+}, \\
& V^{+} \cdot V^{-} \subseteq V^{-} .
\end{aligned}
$$
Hence, the $+1$ eigenspace $V^{+}$ is an euclidean Jordan subalgebra of $V$ with the same identity element $\bf 1$. 
Note that if $V$ itself is already euclidean, then $\alpha=\mathrm{id}_V$ is a Cartan involution, and indeed    the only Cartan involution. In this case, $V^{+}=V$ and $V^{-}=0$.

We denote by $n_+$ and $r_+$ dimension and rank of $V^{+}$and call $r_+$ the {\it split rank} of $V$. The constants $n_+$ and $r_+$ only depend on the isomophism class of the Jordan algebra $V$, 
not on the choice of the involution  $\alpha$.

An element $c\in V$ is said to be  {\it idempotent} if $c\cdot c=c$. Two idempotents $c_1$ and $c_2$ are called {\it orthogonal} if $c_1\cdot c_2=0$.  A non-zero idempotent  is called {\it primitive} if it cannot be written as the  sum of two non-zero orthogonal  idempotents.    

Every  set $\{c_1,\ldots,c_k\}$ of orthogonal primitive idempotents in $V_+$ with the additional condition $c_1+ \cdots+c_k={\bf 1}$ is called a {\it Jordan frame} in $V_+.$ By \cite[Theorem III.1.2]{FK}  
the cardinal   of a  Jordan frame equals  to the rank $r_+$ of $V_+$ and every two Jordan frames are conjugate by an automorphism of $V_+$ (see \cite[Corollary IV.2.7]{FK}).

Fix a Jordan frame $\{c_1,\ldots,c_{r_+}\}$ in $V_+$.  By \cite[Proposition III.1.3]{FK} the spectrum of the multiplication operator $L(c_k)$ by $c_k$  is $\{0,\frac{1}{2},1\}$. Further, the operators  $L(c_1), \ldots, L(c_{r_+})$ commute and therefore are simultaneously diagonalizable. This gives  the following  Peirce decomposition 
\begin{equation*}
V=\bigoplus_{1\leq i\leq j\leq r_+} V_{ij},
\end{equation*}
where
$$\begin{array}{rll}
V_{ii}&:=V(c_i,1), & \quad1\leq i\leq r_+,\\
V_{ij}&:=V(c_i,\frac{1}{2})\cap V(c_j,\frac{1}{2}),& \quad 1\leq i<j\leq r_+.
\end{array}$$
Here  $V(c,\lambda)$ denotes the eigenspace of $L(c)$ corresponding to the eigenvalue $\lambda.$
Since the operators $L(c_k)$, for $0\leq k\leq r_+$, are symmetric with respect to the inner product  \eqref{ds}, then  the above  direct sum  is orthogonal.

Denote by $d$ the common dimension of the subspaces $V_{ij}$ ($i<j$) 
and by $e+1$ the common dimension of the subalgebras $V_{ii}.$   Then, the dimension $n$ of $V$ satisfies  
$$ n=r_+(e+1)+ {d\over 2}{r_+(r_+-1)} .$$


 The Jordan algebra $V$ is called {\it split}  if $V_{ii}=\mathbb{R}c_i$ for every $1\leq i\leq r_+$ (equivalently $e=0$). Otherwise, $V$ is called {\it non-split}. By \cite{Helwig}, if $V$ is split  then $r=r_+$, otherwise $r=2r_+.$ We pin down   that every euclidean Jordan algebra is split.      
  There is  a classification of simple real Jordan algebra  given in  \cite{Helwig}  and \cite{Loos}. We refer the reader to \cite[Appendix A]{bck1}  for the complete  list.  More precisely, we have four types of simple real Jordan algebras :
 
 \begin{tabular}{lll}
 Type I &:& $V$ is euclidean.\\ 
 Type II  &: & $V$ is split non-euclidean.\\ 
 Type III &:&  $V$ is  non-split  with no complex structure.\\
 Type IV &:&  $V$ is  non-split  with a complex structure.
 \end{tabular}
\\

 Notice that every simple real Jordan algebra   is either a real form of a simple complex Jordan algebra (type I, II, III),  or  
 a simple complex Jordan algebra viewed as a real one (type IV).
 


The {structure group} $\Str(V)$ of $V$ can be thought of  as the subgroup of transformations $g\in \GL(V)$   such that for every invertible $x\in V^\times$, 
 there exists a nonzero real number $\chi(g)$ for which
$$
{\bf det}(g x)=\chi(g) {\bf det}(x) .
$$
The map $g \mapsto \chi(g)$ is a character of $\Str(V)$ which is a reductive Lie group.

For $a\in V$, denote by $n_a$ the translation $x\mapsto x+a$. Let 
\[N:=\{ n_a:\ a\in V\}\]
 be the abelian Lie group of all translations.
The   {conformal group} $\Co(V)$ of $V$ is by definition  the group of rational transforms of $V$ generated by $\Str(V)$, $N$  
and the inversion   $\imath: x \mapsto-x^{-1}.$ It can be shown that $\Co(V)$ is a  simple Lie group (see \cite{ber1, ber2}). 
A transformation $g \in \Co(V)$ is conformal in the sense that, at each point $x$, where $g$ is well defined, 
its differential $(D g)_x$ belongs to the structure group $\Str(V)$. In our context, it turned out to be more convenient to work
 with a group $G,$ which is locally isomorphic to $\Co(V ),$  instead of $\Co(V )$ itself. More precisely, we will denote by $\Str(V)^+$ 
 the subgroup of $\Str(V)$ defined by \[\Str(V)^+= \{ \ell\in \Str(V):\quad \chi(\ell)>0\}.\]
 Define the {proper conformal group} $\Co(V)^+$ to be the group generated by $\Str(V)^+,$ $N$ and the inversion $\imath$.

For any $g\in \Co(V)^+$ and $x\in V_g,$ open dense subset of $V$ where $g$ is defined, the differential $Dg(x)\in \Str(V)^+,$ 
and therefore we may define $c(g,x)=\chi(Dg(x))^{-1}.$ In particular,
\begin{itemize}
\item[(i)] For any $\ell\in \Str(V)^+$ and $x\in V$
\[c(\ell, x) = \chi(\ell)^{-1}.
\]
\item[(ii)]  For any $v\in V$ and $x\in V$
\[c(n_v,x) = 1.
\]
\item[(iii)]  For any $x\in V^\times$,
\[c(\imath, x) = {\bf det}(x)^2.
\]
\end{itemize}

In \cite[Proposition 6.1]{bck1} we proved that there exists a polynomial $p_g\in \mathcal P(V)$ such that for every $x\in V_g,$ 
$c(g,x)=p_g(x)^2$ where $p_g$ is unique, up to a sign $\{\pm\}.$ This allows us to construct the twofold covering group 
\begin{equation}\label{til}  G = \Big\{(g, p_g) \in \Co(V)^+ \times  \mathcal{P}(V): p_g(x)^2 = c(g,x)\, \forall x \in V\Big\}.\end{equation}
  The group $G$ is a simple Lie
group which acts rationally on $V$. The  subgroup  $P$ of  affine transformations in $G$, which is a twofold covering  of $\Str(V)^+ \ltimes N$, is in fact a parabolic subgroup of $G$. 
 Let $\bar{P}$ be the opposite parabolic subgroup to $P$. Then the flag variety $\mathcal X:=G / \bar{P}$, which is compact, is the 
 conformal  compactification of $V$ and the map $V \simeq N \longrightarrow G / \bar{P}$ gives the embedding  of $V$ in its compactification. 

For a given $\widetilde g = (g, p_g)\in G$ and $x\in V_g$, we set $\widetilde g(x) := g(x)$. Further,  we will use the following notation 
\begin{equation}\label{a}  
a(\widetilde g, x) := p_g(x),\quad x\in V.
\end{equation}
By  abuse of notation, we will denote elements of $G$ without tilde.

\section{Functional equation of Zeta functions}
Let $V$ be a simple real Jordan algebra with Jordan determinant ${\bf det}.$ We mention that if $V$ is of type I or II, then ${\bf det}$
 takes both positive and negative values, while if $V$ is of type III or IV, then ${\bf det}$ takes only  positive values. For $x\in V^\times$ and 
 $s\in \mathbb C,$  we introduce ${\bf det}(x)^{s, \varepsilon}$ by 
$$ {\bf det}(x)^{s, \varepsilon}=
\begin{cases}
 |{\bf det}(x)|^s & \text{for}  \;\varepsilon =+ \cr 
 {\rm sign}({\bf det}(x)) |{\bf det}(x)|^s & \text {for}\; \varepsilon =-.\cr 
 \end{cases}$$ 

 Let $\mathcal S(V)$ be the Schwartz space of $V$ and let $\mathcal S'(V)$ be  its dual, the space of tempered distributions  on $V.$  For $f\in \mathcal S(V)$ and $(s,\varepsilon)\in \mathbb C\times \{\pm\},$ we set 
\begin{equation}\label{eq1} Z_\varepsilon (f, s) =\int_V f(x) {\bf det}(x)^{s,\varepsilon} dx,\end{equation}
 where $dx$ is the Lebesgue measure on $V$ with respect to the non-degenerate bilinear form \eqref{1}.
The  integral in \eqref{eq1} is convergent for $\Rel s>0,$  the analytic function 
$Z_\varepsilon(f,s)$ has a meromorphic continuation with respect to $s$ to the whole plane $\mathbb C,$ and the map $f\mapsto Z_\varepsilon (f, s) $ is a tempered distribution on $V,$ called a Zeta distribution. For the details, we refer the reader to  \cite{bar, bop, kaya, bck1} when $\varepsilon =+$ and to \cite{bck1} when $\varepsilon =-.$

 The purpose of this section is to give an explicit expression for the Fourier transform of the Zeta distribution for the four types of simple real Jordan algebras. Below, the Fourier transform of $f\in \mathcal S(V)$ is defined by 
 \begin{equation*} 
\widehat{f}(x)=\int_V e^{2i\pi\langle x, y\rangle}f(y)dy,
\end{equation*} 
which we extend it by duality to  the space    of tempered distributions in the standard way. 

For $s\in \mathbb C,$ one defines the gamma function associated to $V$ by 
 \begin{equation*} 
   \Gamma_V(s):=\prod_{k=1}^{r_+}\Gamma\left(\frac{s}{2}-(k-1)\frac{d}{4}\right),
   \end{equation*}
where $r_+$ denotes the split rank of $V.$ Now we are ready to state the functional equation for the Zeta distribution when $V$ is of type II $\not \cong \mathbb R^{p,q},$   type III and type IV. 
\begin{theorem}
The tempered distribution $f\mapsto Z_\varepsilon(f,s)$ satisfies the following functional equation 
  \begin{equation*} 
  Z_\varepsilon(\widehat f, {s })=c(s,\varepsilon) Z_\varepsilon\big(f,-s-\frac{n}{r} \big),
  \end{equation*}
  where
$$
 c(s,\varepsilon)=
  \begin{cases}   
 \displaystyle  \pi^{-rs-\frac{n}{2}}\frac{\Gamma_V\left(s+\frac{n}{r}\right)}{\Gamma_V(-s)} & \text{for}  \;\varepsilon=+ \; {\text and} \, V  \,\text{of type II}\;  \not \cong \mathbb R^{p,q}  \\
    \displaystyle i^r\pi^{-rs-\frac{n}{2}}  \frac{\Gamma_V(s+1+\frac{n}{r})}{\Gamma_V(-s+1)}& \text{for}\;  \varepsilon=-  \; {\text and} \,V \,\text{of type II}\; \not \cong \mathbb R^{p,q}   \\ 
  \displaystyle   \pi^{-rs-\frac{n}{2}}\frac{\Gamma_V\left(2s+\frac{2n}{r}\right)}{\Gamma_V(-2s)} &\text{for}\;  \varepsilon=\pm   \; {\text and} \,V\,\text{of type III or IV}.\\
  \end{cases}
$$
\end{theorem}

When $\varepsilon =+,$ the above theorem goes back to \cite[Theorem 4.4]{bar} (recall that if $V$ is of type  III or IV, then ${\bf det}(x)^{s,+}$ and ${\bf det}(x)^{s,-}$ coincide). In the notation of \cite{bar}, our $|{\bf det} |$  is equal to $\nabla$ when $V$ is of type II, and is equal to $\nabla^2$ when $V$ is of type III or IV.  For the case $\varepsilon =-$ and $V$ is of type II, the above functional equation was established by the three authors in \cite[Theorem 5.2]{bck1}, where we were able to link the Fourier transforms of    $Z_-$ and $Z_+ $ by means of the following Bernestein identity 
$$\det\left(\frac{\partial}{\partial x}\right) {\bf det}(x)^{s+1,+} = \prod_{k=0}^{r-1} \left(s+1+k\frac{d}{2}\right) {\bf det}(x)^{s, -}     $$
 (see \cite[Proposition 3.4]{bck1}).
 
To finish with real simple Jordan algebras of type II, let us consider the case $V=\mathbb R^{p,q},$ for $p\geq 2$ and $q\geq 1.$     
For $x=(x_1, \ldots, x_p, x_{p+1}, \ldots, x_n)\in V,$ with $n=p+q,$ the  determinant is given by ${\bf det}(x)=P(x):=x_1^2+\cdots+x_p^2-x_{p+1}^2-\cdots-x_n^2.$  Here it is customary to define the Fourier  transform of $f\in \mathcal S(\mathbb R^{p,q})$ as the integral of $f$ against the kernel $e^{i \langle x, y\rangle}.$ 

Introduce the polynomial functions $P_+(x)=P(x) \chi_{\{\small P>0\}}(x)$ and $P_-(x)=-P(x) \chi_{\{\small P<0\}}(x).$ Their Fourier transforms as tempered distributions are given in \cite[(2.8) and (2.9)]{GS}. Noticing that  ${\bf det}(x)^{s, +} =P_+(x)^s +P_-(x)^s$ and ${\bf det}(x)^{s, -}=P_+(x)^s-P_-(x)^s,$ one then  can easily deduce the statement belwo:   
 \begin{theorem}\label{E-F-Z-pq} Assume that $V=\mathbb R^{p,q}$ with $p\geq 2$ and $q\geq 1.$ 
  For every $s\in \mathbb C$, the following system of functional equations holds 
 \begin{equation*}
 \left[\begin{matrix} Z_+(\widehat f, s)\\ Z_-(\widehat f,  s) \end{matrix}\right]= 
 \gamma(s){\mathbf A}(s)
\left[ \begin{matrix}
 Z_+\big(f, -s-\frac{n}{2}\big) \\Z_-\big(f, -s-\frac{n}{2}\big)
 \end{matrix}\right]
 \end{equation*}
 where 
 $$\mathbf{A}(s)=
  \left[\begin{array}{lr}
 \displaystyle \cos\frac{(p-q)\pi}{4} \left(\sin \frac{n\pi}{4}-\sin(s+\frac{n}{4}) \pi  \right)
& \displaystyle 
\sin\frac{(p-q)\pi}{4} \left(\cos(s+\frac{n}{4}) \pi- \cos \frac{n\pi}{4}\right)\\
\displaystyle 
 \sin\frac{(p-q)\pi}{4} \left(\cos(s+\frac{n}{4}) \pi+ \cos \frac{n\pi}{4}\right)
& \displaystyle 
-\cos\frac{(p-q)\pi}{4} \left(\sin \frac{n\pi}{4}+\sin(s+\frac{n}{4}) \pi  
  \right)
\end{array}\right]
$$    and
\begin{equation*} \gamma(s) = 2^{2s+n}\pi^{\frac{n}{2}-1} \Gamma(s+1)\Gamma\big(s+\frac{n}{2}\big).\end{equation*}
 \end{theorem}

 Next we turn our attention to the case where $V$ is an Euclidean Jordan algebra, i.e. $V$ is of type I. Our main result depends heavily on the paper \cite{FS} by Faraut and Satake, where the authors proved a system of functional equations for Zeta distributions of type $Z_+$ defined  on orbits  $\Omega_\ell$ consist of elements in $V^{\times}$ of signature $(r-\ell, \ell),$ where $r$ is the rank of $V.$ We may think of our result as a far reaching generalization of Faraut-Satake's result by considering the Zeta distributions $Z_+$ and $Z_-$ defined over the whole $V^\times.$
 
 The set of invertible elements  $V^\times$ in $V$ decomposes into the disjoint union of $r+1$ open $G_\circ$-orbits 
 $$V^\times=\bigcup_{\ell=0}^r \Omega_\ell,$$ where $\Omega_\ell$ is the set of elements of signature $(r-\ell, \ell).$ Here $G_\circ$ is the identity connected component of $G=\Str(V).$ In particular, $\Omega_0,$ the $G_\circ$-orbit of the unit element $\bf 1,$ is a self-dual homogeneous cone. We shall simply denote $\Omega_0$  by $\Omega.$ 
 
 For $f\in \mathcal S(V),$ in \cite{FS} the authors defined the following system of   Zeta integrals  
 $$Z_{\ell}(f,s)=\int_{\Omega_\ell} f(x)|{\bf det}(x)|^sdx,\qquad 1\leq \ell\leq r. $$
The above integral converges  for $\Rel s>0$ and has a meromorphic continuation to the whole plane  $\mathbb{C}$. Their main result is the following system of functional equations 
\begin{equation}\label{F-Eq-Eucl}
Z_\ell \big(\widehat{f},s-\frac{n}{r}\big)=(2\pi)^{-rs}{\rm e}^{i\pi {rs\over 2}}
 \Gamma_{\Omega}(s)\sum_{\kappa=0}^ru_{\ell, \kappa}(s)Z_{\kappa}(f,-s),
\end{equation}
where $u_{\ell,\kappa}(s)$ are polynomials in  ${\rm e}^{-i\pi s}$  of degree at most $r$ (see \cite[(23)]{FS}). Above, $\Gamma_{\Omega}$ is the so-called Gindikin gamma function, 
 $$ 
 \Gamma_{\Omega}(s)
 = (2\pi)^{\frac{n-r}{2}}\prod_{j=1}^r\Gamma\left(s-(j-1)\frac{d}{2}\right).
 $$

For $x\in \Omega_\ell,$ one has ${\bf det}(x)=(-1)^\ell |{\bf det}(x)|.$ Thus, we may express the local Zeta integrals 
$Z_\varepsilon(f, s)$ defined by \eqref{eq1} as following 
\begin{equation}\label{Z+-Z-and Zj}
 \begin{array}{l @{\; =\;} l}
Z_{+}(f,s)&  \displaystyle \int_V f(x) |{\bf det}(x)|^sdx= \sum_{\ell=0}^r Z_\ell(f,s),\\
Z_{-}(f,s)&  \displaystyle \int_V f(x) \mathrm{sgn}({\bf det}(x))|{\bf det}(x)|^sdx= \sum_{\ell=0}^r (-1)^\ell  Z_\ell(f,s).
 \end{array}
 \end{equation}
 Using the system \eqref{F-Eq-Eucl}
 of functional equations, we obtain 
  \begin{equation}\label{eq-f-Z+Zj}
 Z_{ +}\big( \widehat{f}, s-\frac{n}{r} \big) =(2\pi)^{-rs}{\rm e}^{i\pi {rs\over 2}}
 \Gamma_{\Omega}(s) \sum_{\kappa=0}^r\left(\sum_{\ell=0}^r u_{\ell,\kappa}(s)\right) Z_\kappa(f,-s),
 \end{equation}
 and
 \begin{equation}\label{eq-f-Z-Zj}
 Z_{ -}\big(\widehat{f}, s-\frac{n}{r}\big)= (2\pi)^{-rs}{\rm e}^{i\pi {rs\over 2}}
 \Gamma_{\Omega}(s)\sum_{\kappa=0}^r\left(\sum_{\ell=0}^r (-1)^\ell u_{\ell,\kappa}(s)\right) Z_\kappa(f,-s).
 \end{equation}
Next we shall compute the  finite sums $\sum_{\ell }  u_{\ell,\kappa}(s)$ and $\sum_{\ell }  (-1)^\ell u_{\ell,\kappa}(s). $  In \cite{FS}, the authors proved that for any $y\in \mathbb R, $
 \begin{equation}\label{E-eq-f-FS}
\sum_{\ell=0}^r y^\ell u_{\ell, \kappa}(s)=\xi^{-(r-\kappa)}P_\kappa(\xi {\rm e}^{-i\pi s},y)P_{r-\kappa}(1,\xi {\rm e}^{-i\pi s} y), 
\end{equation} 
where $\xi:={\rm e}^{i{\pi\over 2} d(r+1)}$ and 
\begin{equation}\label{coef-mat-FS}  
P_\kappa(a,b)=\begin{cases}
(a+b)^\kappa & \text{if $d$ is even},\\
(a+b)^{ \lfloor  \frac{\kappa}{2} \rfloor }(b-a)^{\kappa-\lfloor  \frac{\kappa}{2} \rfloor} &\text{if $d$ is odd}\, .
\end{cases}
\end{equation} 

First, we consider the case where $d$ is even. We distinguish two cases: 
$$ \begin{array}{rl}
\text{ Case (a):} &d\equiv0\ (\mathrm{mod}\ 4) \text{ or } d\equiv2\ (\mathrm{mod}\ 4) \text{ and } r \text{ odd},\\ 
\text{Case (a'):}&  d\equiv 2\ (\mathrm{mod}\ 4) \text{ and }  r \text{ even}
\end{array}
$$
 Then one has 
 $$ 
\sum_{\ell=0}^r y^\ell u_{\ell,\kappa}(s)=\begin{cases}
(y+e^{-i\pi s})^{\kappa} (ye^{-i\pi s} +1)^{r-\kappa}  & \text{ in Case (a)}\\
(y-e^{-i\pi s})^{\kappa} (ye^{-i\pi s} -1)^{r-\kappa}  & \text{ in Case (a')}. \end{cases}
  $$
In Case (a), we deduce that 
$$\sum_{\ell=0}^r  u_{\ell,\kappa}(s) = 2^r e^{-i{\pi \over 2} rs } \cos^r\left(\frac{\pi s}{2}\right) ,$$
and 
$$\sum_{\ell=0}^r  (-1)^\ell u_{\ell,\kappa}(s) = (2 i)^r (-1)^\kappa e^{-i{\pi \over 2} rs }  \sin^r\left(\frac{\pi s}{2}\right) .$$
Hence, in Case (a), the functional equations  \eqref{eq-f-Z+Zj} and \eqref{eq-f-Z-Zj} become 
 \begin{equation} 
 Z_{ +}\big( \widehat{f}, s-\frac{n}{r} \big) =2^r (2\pi)^{-rs} 
 \Gamma_{\Omega}(s)   \cos^r\left(\frac{\pi s}{2}\right) Z_+(f,-s),
 \end{equation}
 and
 \begin{equation} 
 Z_{ -}\big(\widehat{f}, s-\frac{n}{r}\big)=(2i)^r (2\pi)^{-rs}  \Gamma_{\Omega}(s)  \sin^r\left(\frac{\pi s}{2}\right) 
  Z_-(f,-s).
 \end{equation}
Similarly, in Case (a') one obtains  
$$\sum_{\ell=0}^r  u_{\ell,\kappa}(s) = (-1)^\kappa (2i)^r       e^{-i{\pi \over 2} rs } \sin^r\left(\frac{\pi s}{2}\right) ,$$
and 
$$\sum_{\ell=0}^r  (-1)^\ell u_{\ell,\kappa}(s) =2^r   e^{-i{\pi \over 2} rs }  \cos^r\left(\frac{\pi s}{2}\right) ,$$
which yield to 
$$ 
 Z_{ +}\big( \widehat{f}, s-\frac{n}{r}\big)  =(2i)^r (2\pi)^{-rs}   \Gamma_{\Omega}(s)\sin^r\left(\frac{\pi s}{2}\right)Z_{-}(f,-s),
 $$
and  
$$ 
 Z_{ -}\big( \widehat{f}, s-\frac{n}{r}\big)=2^r(2\pi)^{-rs} \Gamma_{\Omega}(s)  \cos^r\left(\frac{\pi s}{2}\right)Z_{+}(f,-s).
 $$
Hence, we have the following statement:
\begin{theorem} Assume that $V$ is an euclidean Jordan algebra of type Case (a) or Case (a'). The Zeta integrals \eqref{eq1} satisfy the following system of functional equations: 
\begin{itemize} 
\item {\rm Case (a):}
$$
\left[\begin{matrix}
Z_{+}(\widehat f, s) \\
Z_{-}(\widehat f, s)
\end{matrix}\right]
= 2^r (2\pi)^{-r(s+{n\over r})} \Gamma_{\Omega}\Big(s+\frac{n}{r}\Big) \mathbf{A}(s) 
\left[
\begin{matrix}
Z_+ \big(f, -s-\frac{n}{r}\big) \\
Z_-\big(f, {-s-\frac{n}{r}}\big)
\end{matrix}\right],
$$
where
$$\mathbf{A}(s)=\left[\begin{matrix}
\displaystyle \cos^r\Big(\frac{\pi }{2}(s+\frac{n}{r})\Big) & 0\\
0&\displaystyle i^r \sin^r\Big(\frac{\pi }{2}(s+\frac{n}{r})\Big)
\end{matrix}\right].$$
\item {\rm Case (a'):}
$$
\left[\begin{matrix}
Z_{+}(\widehat f, s) \\
Z_{-}(\widehat f, s)
\end{matrix}\right]
= 2^r (2\pi)^{-r(s+{n\over r})} \Gamma_{\Omega}\Big(s+\frac{n}{r}\Big) \mathbf{A}(s) 
\left[
\begin{matrix}
Z_+ \big(f, -s-\frac{n}{r}\big) \\
Z_-\big(f, {-s-\frac{n}{r}}\big)
\end{matrix}\right],
$$
where
$$\mathbf{A}(s)=\left[\begin{matrix}
\displaystyle i^r \sin^r \Big(\frac{\pi }{2}(s+\frac{n}{r})\Big) & 0\\
0&\displaystyle \cos^r\Big(\frac{\pi }{2}(s+\frac{n}{r})\Big)
\end{matrix}\right].$$
\end{itemize}
\end{theorem} 
Next we consider the case where $d$ is odd. According to the classification theory of euclidean Jordan algebras, we have the following two possibilities: 
$$
 \begin{array}{rl}
    \text{Case (b):} &     d \text{ odd and }   r=2,\\
      \text{Case (c):}  & d=1  \text{ and arbitrary  } r .\\
 \end{array}
 $$

In Case (b), by \eqref{E-eq-f-FS} we have 
 $$ \sum_{\ell=0}^2 y^\ell u_{\ell,\kappa}(s)= \begin{cases} 
     1+e^{-2\pi i s}y^2 & \text{for } \kappa=0\\
 e^{-\pi i s}-i^{d}(1-e^{-2\pi i s})y + e^{-\pi i s}y^2 & \text{for } \kappa=1 \\
  e^{-2\pi i s}+y^2, & \text{for } \kappa=2
   \end{cases}$$
 for any  real number $y$.  Therefore, for $d\equiv 1\ (\mathrm{mod}\ 4)$,   we get 
  $$ \sum_{\ell=0}^2 (\pm 1)^\ell u_{\ell,\kappa}(s)= \begin{cases} 
2  e^{-\pi i s} \cos(\pi s) & \text{for } \kappa=0\\
2 e^{-\pi i s} \pm 2  e^{-\pi i s} \sin(\pi s) & \text{for } \kappa=1 \\
2  e^{-\pi i s} \cos(\pi s) & \text{for } \kappa=2.
   \end{cases}$$
   The functional equations \eqref{eq-f-Z+Zj} and \eqref{eq-f-Z-Zj} become
  $$ Z_{\pm}\big( \widehat f, s-{n\over 2}\big)= 2 (2\pi)^{-2s} 
 \Gamma_{\Omega}(s)  \Big\{     \cos(\pi s) \big( Z_0(f,-s) +Z_2(f,-s)\big)+ (1\pm  \sin(\pi s) ) Z_1(f,-s)\Big\}
 $$
   On the other hand, by \eqref{Z+-Z-and Zj}, we have 
    $ Z_0(f,-s)+Z_1(f,-s)+Z_2(f,-s)=Z_{ +}(f,-s)$ and $ Z_0(f,-s)-Z_1(f,-s)+Z_2(f,-s)=Z_{ -}(f,-s).$ That is 
   $$ Z_0(f,-s) +Z_2(f,-s) ={1\over 2} \big( Z_{ +}(f,-s) +Z_{ -}(f,-s) \big),$$
   and $$ Z_1(f,-s)= {1\over 2}\big( Z_{ +}(f,-s) - Z_{ -}(f,-s) \big).$$
 By putting  pieces together, we arrive at 
 $$ Z_{+}\big( \widehat f, s-{n\over 2}\big)= 2\sqrt 2  (2\pi)^{-2s} 
 \Gamma_{\Omega}(s) \sin \Big({{\pi s}\over 2} +{\pi\over 4}\Big)  \Big\{     \cos\Big({{\pi s}\over 2}\Big)   Z_+(f,-s)  - \sin\Big({{\pi s}\over 2}\Big)   Z_-(f,-s)\Big\}
 $$
 and 
 $$ Z_{-}\big( \widehat f, s-{n\over 2}\big)= 2\sqrt 2  (2\pi)^{-2s} 
 \Gamma_{\Omega}(s) \cos \Big({{\pi s}\over 2} +{\pi\over 4}\Big)  \Big\{     \cos\Big({{\pi s}\over 2}\Big)   Z_+(f,-s)  + \sin\Big({{\pi s}\over 2}\Big)   Z_-(f,-s)\Big\}.
 $$
 
Similarly, for $d\equiv 3\ (\mathrm{mod}\ 4)$  we have
 $$ \sum_{\ell=0}^2 (\pm 1)^\ell u_{\ell,\kappa}(s)= \begin{cases} 
2  e^{-\pi i s} \cos(\pi s) & \text{for } \kappa=0\\
2 e^{-\pi i s} \mp 2  e^{-\pi i s} \sin(\pi s) & \text{for } \kappa=1 \\
2  e^{-\pi i s} \cos(\pi s) & \text{for } \kappa=2,
   \end{cases}$$
   which lead to the following functional equations 
   $$ Z_{+}\big( \widehat f, s-{n\over 2}\big)= 2\sqrt 2  (2\pi)^{-2s} 
 \Gamma_{\Omega}(s) \cos \Big({{\pi s}\over 2} +{\pi\over 4}\Big)  \Big\{     \cos\Big({{\pi s}\over 2}\Big)   Z_+(f,-s)  + \sin\Big({{\pi s}\over 2}\Big)   Z_-(f,-s)\Big\}
 $$
 and 
 $$ Z_{-}\big( \widehat f, s-{n\over 2}\big)= 2\sqrt 2  (2\pi)^{-2s} 
 \Gamma_{\Omega}(s) \sin \Big({{\pi s}\over 2} +{\pi\over 4}\Big)  \Big\{     \cos\Big({{\pi s}\over 2}\Big)   Z_+(f,-s)  - \sin\Big({{\pi s}\over 2}\Big)   Z_-(f,-s)\Big\}.
 $$
   \begin{theorem} Assume that $V$ is an euclidean Jordan algebra of type Case (b). The Zeta integrals \eqref{eq1} satisfy the following system of functional equations: 
\begin{itemize} 
\item If $d\equiv 1\ (\mathrm{mod}\ 4)$, then
$$
\left[\begin{matrix}
Z_{+}(\widehat f, s) \\
Z_{-}(\widehat f, s)
\end{matrix}\right]
=2\sqrt 2  (2\pi)^{-2(s+{n\over 2})} \Gamma_{\Omega}\Big(s+\frac{n}{2}\Big) \mathbf{A}(s) 
\left[
\begin{matrix}
Z_+ \big(f, -s-\frac{n}{2}\big) \\
Z_-\big(f, {-s-\frac{n}{2}}\big)
\end{matrix}\right],
$$
where
$$\mathbf{A}(s)=\left[\begin{matrix}
\sin\Big(\frac{\pi }{2}(s+\frac{n+1}{2})\Big) \cos\Big(\frac{\pi }{2}(s+\frac{n}{2})\Big) & 
-\sin\Big(\frac{\pi }{2}(s+\frac{n+1}{2})\Big) \sin\Big(\frac{\pi }{2}(s+\frac{n}{2})\Big)\\
\cos\Big(\frac{\pi }{2}(s+\frac{n+1}{2})\Big) \cos\Big(\frac{\pi }{2}(s+\frac{n}{2})\Big)&
\cos\Big(\frac{\pi }{2}(s+\frac{n+1}{2})\Big) \sin\Big(\frac{\pi }{2}(s+\frac{n}{2})\Big)
\end{matrix}\right].$$
\item If $d\equiv 3\ (\mathrm{mod}\ 4)$, then 
$$
\left[\begin{matrix}
Z_{+}(\widehat f, s) \\
Z_{-}(\widehat f, s)
\end{matrix}\right]
= 2\sqrt 2  (2\pi)^{-2(s+{n\over 2})} \Gamma_{\Omega}\Big(s+\frac{n}{2}\Big) \mathbf{A}(s) 
\left[
\begin{matrix}
Z_+ \big(f, -s-\frac{n}{2}\big) \\
Z_-\big(f, {-s-\frac{n}{2}}\big)
\end{matrix}\right],
$$
where
$$\mathbf{A}(s)=\left[\begin{matrix}
\cos\Big(\frac{\pi }{2}(s+\frac{n+1}{2})\Big) \cos\Big(\frac{\pi }{2}(s+\frac{n}{2})\Big) & 
\cos\Big(\frac{\pi }{2}(s+\frac{n+1}{2})\Big) \sin\Big(\frac{\pi }{2}(s+\frac{n}{2})\Big)\\
\sin\Big(\frac{\pi }{2}(s+\frac{n+1}{2})\Big) \cos\Big(\frac{\pi }{2}(s+\frac{n}{2})\Big)&
-\sin\Big(\frac{\pi }{2}(s+\frac{n+1}{2})\Big) \sin\Big(\frac{\pi }{2}(s+\frac{n}{2})\Big)
\end{matrix}\right].$$
\end{itemize}
\end{theorem} 

We close this section by the above mentioned Case (c), that is $d=1$ and $r$ arbitrary. In this case one needs to distinguish the four cases $r\equiv 0, 1, 2$ and $3 \ (\text{mod } 4).$ We shall present briefly the case $r\equiv 0\ (\text{mod } 4).$ For the three remaining cases, we refer the reader to \cite{bck1} for a detailed exposition. 

In view of  \eqref{E-eq-f-FS} and the fact that   $r\equiv 0\ (\text{mod } 4)$, we have 
$$ \sum_{\ell=0}^r (\pm 1)^\ell u_{\ell,\kappa}(s)= \begin{cases} 
(1+e^{-2\pi is})^{r/2}=2^{r/2} \cos^{r/2}(\pi s) e^{-i\pi{{rs}\over 2}} & \text{for even } \kappa \\
 \mp i(1+e^{-2\pi is})^{r/2}=\mp i 2^{r/2} \cos^{r/2}(\pi s) e^{-i\pi{{rs}\over 2}}& \text{for odd } \kappa. 
   \end{cases}$$
Therefore,
 \begin{eqnarray*} &&Z_{\pm}\big( \widehat f, s-{n\over r}\big)\\
 &=&  2^{r/2} (2\pi)^{-rs} 
 \Gamma_{\Omega}(s)  \cos^{r/2}(\pi s)  \Big\{   \sum_{\kappa { \,\rm even}} Z_\kappa(f, -s) \mp i \sum_{\kappa {\,\rm odd}} Z_\kappa(f, -s)  \Big\}\\
 &=&  2^{r/2-1} (2\pi)^{-rs} 
 \Gamma_{\Omega}(s)  \cos^{r/2}(\pi s)  \Big\{  Z_+(f,-s) +Z_-(f,-s) \mp i  (Z_+(f, -s)-Z_-(f,-s))  \Big\}\\
 &=&  2^{r/2-1} (2\pi)^{-rs} 
 \Gamma_{\Omega}(s)  \cos^{r/2}(\pi s)  \Big\{  (1\mp i) Z_+(f,-s) +  (1\pm i) Z_-(f,-s)  \Big\}\\
 &=&  2^{r/2-1/2} (2\pi)^{-rs} 
 \Gamma_{\Omega}(s)  \cos^{r/2}(\pi s)  \Big\{ e^{\mp i{\pi\over 4}} Z_+(f,-s) + e^{\pm i{\pi\over 4}}   Z_-(f,-s)  \Big\}.
 \end{eqnarray*}

The remaining cases $r\equiv  1, 2$ and $3 \ (\text{mod } 4) $ can be treated in a similar way (see \cite[pages 2314--2317]{bck1} for an exhaustive   exposition). To  state the complete result for Case (c) we need to introduce the following tempered distributions, 
 \begin{equation}\label{Z-odd-even}
 Z^{\text{e}}(f,s)=\sum_{\kappa=0}^{\lfloor \frac{r}{2}\rfloor} (-1)^\kappa Z_{2\kappa}(f,s)\quad\text{and }\quad  Z^{\text{o}}(f,s)  =\sum_{\kappa=0}^{\lfloor \frac{r-1}{2}\rfloor}  (-1)^\kappa Z_{2\kappa+1}(f,s).
 \end{equation}

\begin{theorem} Assume that $V$ is an euclidean Jordan algebra with $d=1$ and of arbitrary rank $r.$ 
The Zeta integrals \eqref{eq1} satisfy the following system of functional equations: 
\begin{itemize} 
\item If $r\equiv 0\ (\mathrm{mod}\ 4)$, then
$$
\left[\begin{matrix}
Z_{+}(\widehat f, s) \\
Z_{-}(\widehat f, s)
\end{matrix}\right]
=e^{i{\pi\over 4}} 2^{{r-1}\over 2}  (2\pi)^{-r(s+{n\over r})} \Gamma_{\Omega}\Big(s+\frac{n}{r}\Big)  \cos^{r\over 2}\Big(\pi (s+{n\over r})\Big) \mathbf{A}(s) 
\left[
\begin{matrix}
Z_+ \big(f, -s-\frac{n}{r}\big) \\
Z_-\big(f, {-s-\frac{n}{r}}\big)
\end{matrix}\right],
$$
where
$$\mathbf{A}(s)=\left[\begin{matrix}
-i & 1 \\
1&-i
\end{matrix}\right].$$
\item If $r\equiv 2\ (\mathrm{mod}\ 4)$, then
$$
\left[\begin{matrix}
Z_{+}(\widehat f, s) \\
Z_{-}(\widehat f, s)
\end{matrix}\right]
=e^{i{\pi\over 4}} 2^{{r-1}\over 2}  (2\pi)^{-r(s+{n\over r})} \Gamma_{\Omega}\Big(s+\frac{n}{r}\Big)  \cos^{r\over 2}\Big(\pi (s+{n\over r})\Big) \mathbf{A}(s) 
\left[
\begin{matrix}
Z_+ \big(f, -s-\frac{n}{r}\big) \\
Z_-\big(f, {-s-\frac{n}{r}}\big)
\end{matrix}\right],
$$
where
$$\mathbf{A}(s)=\left[\begin{matrix}
1&-i   \\
 -i&1
\end{matrix}\right].$$

\item If $r\equiv 1\ (\mathrm{mod}\ 4)$, then 
$$
\left[\begin{matrix}
Z_{+}(\widehat f, s) \\
Z_{-}(\widehat f, s)
\end{matrix}\right]
= (2i)^{\lfloor\frac{r}{2} \rfloor}  (2\pi)^{-r(s+{n\over r})} \Gamma_{\Omega}\Big(s+\frac{n}{r}\Big)  \sin^{\lfloor\frac{r}{2} \rfloor}\Big(\pi (s+{n\over r})\Big) \mathbf{A}(s) 
\left[
\begin{matrix}
Z^{\rm e} \big(f, -s-\frac{n}{r}\big) \\
Z^{\rm o}\big(f, {-s-\frac{n}{r}}\big)
\end{matrix}\right],
$$
where
$$\mathbf{A}(s)=\left[\begin{matrix}
\cos\Big({\pi\over 2}  (s+{n\over r})\Big) &\cos\Big({\pi\over 2}(s+{n\over r})\Big)    \\
i \sin\Big({\pi\over 2} (s+{n\over r})\Big) &-i \sin\Big({\pi\over 2} (s+{n\over r})\Big)
\end{matrix}\right].$$

\item If $r\equiv 3\ (\mathrm{mod}\ 4)$, then 
$$
\left[\begin{matrix}
Z_{+}(\widehat f, s) \\
Z_{-}(\widehat f, s)
\end{matrix}\right]
= (2i)^{\lfloor\frac{r}{2} \rfloor}  (2\pi)^{-r(s+{n\over r})} \Gamma_{\Omega}\Big(s+\frac{n}{r}\Big)  \sin^{\lfloor\frac{r}{2} \rfloor}\Big(\pi (s+{n\over r})\Big) \mathbf{A}(s) 
\left[
\begin{matrix}
Z^{\rm e} \big(f, -s-\frac{n}{r}\big) \\
Z^{\rm o}\big(f, {-s-\frac{n}{r}}\big)
\end{matrix}\right],
$$
where
$$\mathbf{A}(s)=\left[\begin{matrix}
i \sin\Big({\pi\over 2} (s+{n\over r})\Big) &-i \sin\Big({\pi\over 2} (s+{n\over r})\Big)\\
\cos\Big({\pi\over 2}  (s+{n\over r})\Big) &\cos\Big({\pi\over 2}(s+{n\over r})\Big)    
\end{matrix}\right].$$

\end{itemize}
\end{theorem}

\section{Construction of the source operator}
  
Recall that  $\mathcal X=G/\bar P$ be the conformal completion of $V$, where $G$ is the double covering \eqref{til} of $\Co(V)^+.$ The characters of $\bar P$ are parametrized by $(\lambda, \varepsilon)\in\mathbb C\times \{\pm\}$. To a character $\chi_{\lambda, \varepsilon}$,  there corresponds a line bundle over $\mathcal X$ and let $\mathcal H_{\lambda, \varepsilon}$ be the corresponding space of smooth sections of this bundle, which can be realized as the space of smooth functions $f : G\longrightarrow \mathbb C$ satisfying
\[ f(g\bar p) = \chi_{\lambda, \varepsilon}(\bar p)^{-1}f(g),
\] for each $g\in G \text{ and } \bar p\in \bar P.$
The  formula 
$$
\big(\pi_{\lambda, \varepsilon}(g)f\big)(x) = a(g^{-1},x)^{-\lambda, \varepsilon} f\big(g^{-1}(x)\big) 
$$
defines a smooth representation $\pi_{\lambda, \varepsilon}$ of $G$ on $\mathcal H_{\lambda, \varepsilon}$, where $a(g,x)$ is as in \eqref{a}.

Consider two pairs  $(\lambda,\varepsilon),\,(\mu,\eta)\in \mathbb C\times \{\pm\}$. The main idea of the \emph{source operator method}, inspired by the $\Omega$-process (see \cite{c23}), is to construct  a differential operator \[E_{(\lambda,\varepsilon),\, (\mu,\eta)} : \ \mathcal H_{\lambda,\, \varepsilon} \otimes \mathcal H_{\mu,\, \eta} \longrightarrow \mathcal H_{\lambda+1,\, -\varepsilon} \otimes \mathcal H_{\mu+1,\, -\eta}\] 
which is $G$-covariant with respect to  $(\pi_{\lambda,\varepsilon}\otimes \pi_{\mu,\eta}\,,\,\pi_{\lambda+1, -\varepsilon}\otimes \pi_{\mu+1,-\eta})$.

The source operator is obtained by combining a natural \emph{multiplication  operator} $M$ and the classical \emph{Knapp-Stein intertwining operators}.

To describe the operator $M$, it is convenient to use the noncompact picture. Recall that the injective map $  V\ni v\mapsto n_v$ gives a chart on an open set of $\mathcal X$ which is dense and of full measure. The Jordan determinant $\bf det$ of   $V$ satisfies the following covariance relation 
\begin{equation} \label{covdet}
{\bf det}(g(x)-g(y)) = a(g,x)^{-1} \,{\bf det}(x-y)\, a(g,y)^{-1},
\end{equation}
whenever $g$ is defined at $x$ and $y$, with $x,y\in V$ and  $g\in G$. For $g=\iota$, this is the famous \emph{Hua formula} and the general formula is obtained by verifying it on the generators of $G$ and then by using the cocycle property of  $a(g,x)$.

Let $f\in \mathcal H_{\lambda, \varepsilon}\otimes  \mathcal H_{\mu, \eta}$ and let $f(x,y)$ be its corresponding local expression on $V\times V$. Then, thanks to the covariance relation \eqref{covdet}, the function
\[ {\bf det}(x-y) f(x,y)
\]
can be interpreted as the local expression of some $Mf\in \mathcal H_{\lambda-1,-\varepsilon}\otimes \mathcal H_{\mu-1,-\eta}$,
and the operator $M$ so defined satisfies the intertwining relation
\[M\circ \big(\pi_{\lambda,\varepsilon}(g)\otimes \pi_{\mu,\eta}(g)\big)
=\big(\pi_{\lambda-1,-\varepsilon}(g)\otimes \pi_{\mu-1,-\eta}(g) \big)\circ M.\]
The \emph{Knapp-Stein} operators are a meromorphic family of operators 
\[
I_{\lambda,\varepsilon} : \mathcal H_{\lambda,\,\varepsilon}\longrightarrow \mathcal H_{\frac{2n}{r}-\lambda,\, \varepsilon}
\]
which satisfy the intertwining relation
\[I_{\lambda,\, \varepsilon} \circ \pi_{\lambda,\,\varepsilon}(g) = \pi_{\frac{2n}{r} -\lambda,\, \varepsilon}(g)\circ I_{\lambda,\, \varepsilon}\qquad\text{ for any } g\in G.
\]
Now consider the following diagram
$$
\begin{CD}
\mathcal H_{(\lambda,\varepsilon), (\mu,\eta)} @> F_{(\lambda, \varepsilon), (\mu,\eta)}>>  \mathcal H_{(\lambda+1,-\varepsilon),(\mu+1,-\eta)}\\
@V I_{\lambda,\varepsilon}\otimes I_{\mu,\eta} V  V @AA I_{\frac{2n}{r}-\lambda-1, -\varepsilon}\otimes\, I_{\frac{2n}{r}-\mu-1,-\eta}A\\
\mathcal H_{(\frac{2n}{r}-\lambda,\varepsilon), (\frac{2n}{r}-\mu,\eta)} @>M> >\mathcal H_{(\frac{2n}{r}-\lambda-1, -\varepsilon),(\frac{2n}{r}-\mu-1,-\eta)}
\end{CD}
$$
That is,
\begin{equation}\label{defFlambdamu}
F_{(\lambda,\varepsilon),(\mu, \eta)}  = (I_{\frac{2n}{r}-\lambda-1,-\varepsilon}\otimes I_{\frac{2n}{r}-\mu-1, -\eta})\circ M\circ (I_{\lambda, \varepsilon} \otimes I_{\mu, \eta}).
\end{equation}
By construction,   $F_{(\lambda, \varepsilon),(\mu,\eta)}$ is covariant with respect to  $(\pi_{\lambda,\varepsilon}\otimes \pi_{\mu,\eta}\,,\,\pi_{\lambda+1, -\varepsilon}\otimes \pi_{\mu+1,-\eta})$ and depends meromorphically on $(\lambda, \mu)$.

The main theorem can now be stated.
\begin{theorem} \label{Flambdamu}
The operator $F_{(\lambda, \epsilon), (\mu, \eta)}$ is a \emph{differential operator} on $\mathcal X\times \mathcal X$. Moreover, after a normalization by a meromorphic function in $(\lambda, \mu)$, the operator depends polynomially on $(\lambda, \mu)$.
\end{theorem}

The proof that $F_{(\lambda, \varepsilon),(\mu,\eta)}$ is a {differential operator} is done in the noncompact picture. We keep the same notation for the local expression of the operators involved in the definition of $F_{(\lambda,\varepsilon),(\mu, \eta)} $. The proof uses as an essential ingredient the \emph{Fourier transform} on $V\times V$. In the local chart $V\times V$, the operator $M$ is  the multiplication by a polynomial and a Knapp-Stein operator is  a convolution  by a tempered distribution on $V$. Both operators have nice correspondents on the Fourier side. However, the Knapp-Stein operators have singular kernels. So, in order to compose the operators, there are some technical difficulties. To circumvent these difficulties, we use the Schwartz spaces $\mathcal S(V),$ $ \mathcal S(V\times V)\simeq \mathcal S(V)\otimes \mathcal S(V)$ and their dual spaces, namely the spaces of tempered distributions $\mathcal S'(V)$ and  $\mathcal S'(V\times V)$. The Schwartz spaces are stable under the infinitesimal action of the representations $\pi_{\lambda, \varepsilon}$ and   by the multiplication operator $M$, while the Knapp-Stein operators map $\mathcal S(V)$ into $\mathcal S'(V)$. So it is possible to consider the composition of $M$ and of \emph{one} Knapp-Stein operator, getting an operator from $\mathcal S(V)$ into $\mathcal S'(V)$.  A (seemingly weaker) form of the identity used for the definition of $F_{(\lambda, \varepsilon), (\mu, \eta)}$ can be used. In fact, the Knapp-Stein operators are generically invertible (i.e. outside of a denumerable set of values of $\lambda$) and the inverse of $I_{\lambda, \varepsilon}$ is (up to a meromorphic factor) the operator $I_{\frac{2n}{r}-\lambda,\varepsilon}$ so that \eqref{defFlambdamu} can be replaced by
\begin{equation}\label{weakdef}
M\circ (I_{\lambda, \varepsilon}\otimes I_{\mu, \eta}) = \kappa(\lambda, \mu) (I_{\lambda+1,-\varepsilon}\otimes I_{\mu+1,-\eta})\circ F_{(\lambda,\varepsilon),(\mu, \eta)},
\end{equation}
where $\kappa(\lambda, \mu)$ is a meromorphic function on $\mathbb C\times \mathbb C$.
As observed in the $\Omega$-process  \cite{c23}, in the noncompact picture the source operator is a differential operator with \emph{polynomial coefficients}. We also pin down that the coefficients do not depend on $(\varepsilon, \eta)$, and hence we omit these indices in the notation. The corresponding operator on the Fourier transform side is also a differential operator with polynomial coefficients. It turns out that the same properties hold in general. From these observations, it  follows that  \eqref{weakdef} is clearly equivalent to  an operator identity on the Fourier transform side, which we will now describe.

To the multiplication operator $M$ corresponds on the Fourier side transform the constant coefficients differential operator $\displaystyle {\bf det}\left(\frac{\partial}{\partial \xi}-\frac{ \partial}{\partial \zeta} \right)$. The kernel of the Knapp-Stein operator is (up to a normalizing factor)
\[{\bf det}( x)^{-\frac{2n}{r}+\lambda, \varepsilon}\ ,
\]
and its Fourier transform was studied in Section 3. After a change of parameters and up to  normalizing factors, the identity \eqref{weakdef} turns out to be equivalent to an identity on the Fourier side, which we state as a theorem.
\begin{theorem} \label{Dst}
 There exists a differential operator $D_{s,t}$ with polynomial coefficients on $V\times V$ such that 
\begin{equation}\label{Fourierversion}
{\bf det}\left( \frac{\partial}{\partial \xi}-\frac{\partial}{\partial \zeta}\right)\circ {\bf det} (\xi)^{s,\varepsilon} \,{\bf det} (\zeta)^{t,\eta} ={\bf det} (\xi)^{s-1,-\varepsilon}{\bf det}( \zeta)^{t-1,-\eta}\circ D_{s,t}.
\end{equation}
Moreover, the coefficients of $D_{s,t}$ are polynomial  in $(s,t)$.
\end{theorem}
Let us sketch a proof of this result, which is simpler than the one given  in \cite{bck1}.
\smallskip

{\bf Step 1.} The first ingredient in the proof is a version of the Taylor formula and of the Leibniz formula, with reference to the \emph{Fischer inner product} on the space of polynomials.
Let $(E, \langle\!\!\langle\cdot,\cdot \rangle\!\!\rangle)$ be a finite dimensional Euclidean vector space. Let $\mathcal P$ be the space of all (real-valued) polynomials on $E$. For any  $p\in \mathcal P$, we let $\displaystyle p\left( \frac{\partial}{\partial x}\right)$ to be the unique constant coefficients differential operator  on $E$ characterized by the following property, valid for any $y\in E$
\[p\left( \frac{\partial}{\partial x}\right) e^{\langle\!\!\langle x,y \rangle\!\!\rangle}= p(y) e^{\langle\!\!\langle x,y \rangle\!\!\rangle}\ .
\]
The same operator will occasionally be denoted  by $ \partial(p)$.

Recall that the Fischer inner product on $\mathcal P$ is given by
\[(p,q)_F = \partial(p)q(0).
\]
The main property of the Fischer product is the following: for $p,q,r\in \mathcal P$
\begin{equation}
(p,qr)_F = (\partial(r)p,q)_F.
\end{equation}

For $k\in \mathbb N$, denote by $\mathcal P_k$ the subspace of all polynomials which are homogeneous of degree $k$. The corresponding decomposition 
\[\mathcal P = \bigoplus_{k=0}^\infty \mathcal P_k
\]
is orthogonal with respect to the Fischer inner product.

Let $p\in \mathcal P$. For any $a\in E$, denote by $p_a$ the polynomial given by
\[p_a(x) = p(x+a)\ .
\]
Let $\mathcal R$ be the subspace of $\mathcal P$ spanned by all $(p_a)_{a\in E}$. It is also the subspace spanned by all partial derivatives of $p$. Choose an orthonormal basis (for the Fischer inner product) $(p_i)_{i\in I}$ of $\mathcal R$. To any element $p_i$ of the basis, associate the polynomial $\breve{p}_i$ defined by
\[\breve{p}_i = \partial (p_i) p.
\]
By definition of $\mathcal R$, $\breve{p}_i$ belongs to $\mathcal R$.

\begin{proposition}\label{propgenTaylor}  
\begin{itemize}
\item[{i)}]  For all $x,y\in E$
\begin{equation}\label{genTaylor}
p(x+y) = \sum_{i\in I} p_i(x) \, \breve{p}_i (y),
\end{equation}
\item[{ii)}]  For all smooth functions $f,g$ on $E$
\begin{equation}\label{genLeibniz}
\partial(p) (fg) = \sum_{i\in I} \partial(\breve{p}_i)f\ \partial(p_i)g.
\end{equation}
\end{itemize}

\end{proposition}

\begin{proof} Fix $y\in E$  and let $p_y (x)= p(x+y)$. As $\mathcal R$ is stable by translations, $p_y$ belongs to $\mathcal R$. As $(p_i)_{i\in I}$ is an orthonormal basis, 
\[p_y(x) = \sum_{i\in I} ( p_y,p_i )_F \  p_i(x)\ .
\]
Next,
\[  ( p_y,p_i)_F       = \partial(p_i)p_y(0) = \big(\partial(p_i) p(\cdot+y)\big)(0)= \big(\partial(p_i) p\big)(0+y) =\big( \partial(p_i)p\big) (y),
\]
and \eqref{genTaylor} follows.

It is enough to prove \eqref{genLeibniz} for $f,g\in \mathcal P$. As all differential operators $\partial(p), \partial(p_i), \partial (\breve{p}_i)$ have constant coefficients, they commute with translations. It is enough to prove the equality at $x=0$. Next,
\[\partial(p) (fg)\, (0) = (p, fg)_F = (\partial(f) p,g)_F.
\]
As $\partial(f) p$ belongs to $\mathcal R$,  $\displaystyle \partial(f) p = \sum_{i\in I} a_i p_i$ where
\[ a_i = (\partial( f)p, p_i)_F= (p, fp_i) = (\partial(p_i) p, f) = (\breve{p}_i,f)_F = \partial(\breve{p_i})f\, (0).
\]
Hence
\[\partial(p) (fg)\, (0) =\sum_{i\in I}a_i(p_i,g)_F= \sum_{i\in I} \partial(\breve{p_i}) f\, (0)\ \partial(p_i)g\, (0).
\]
 \end{proof}
Assume moreover that $p$ is homogenous of degree $k\in \mathbb N$. Then $\mathcal R$ is contained in the space of polynomials of degree less than or equal to $k$, and more precisely,
\[\mathcal R =\bigoplus_{\ell = 0}^k \mathcal R_\ell, \qquad \mathcal R_\ell = \mathcal R\cap \mathcal P_\ell.
\]
Set $d_\ell = \dim \mathcal R_\ell$. As the direct sum is orthogonal w.r.t. the Fischer inner product, the orthogonal basis $(p_i)_{ i\in I}$ can (and will always) be chosen in the form
\[\bigcup_{\ell=0}^k  \ (p_{\ell,i})_{1\leq i\leq d_\ell},
\]
where $(p_{\ell,i})_{1\leq i\leq d_\ell}$ is an orthonormal basis of $\mathcal R_\ell$.
\medskip

{\bf Step 2.} Let us apply these ideas to our context. Let $\mathcal W$ be the subspace of $\mathcal P(V)$ generated by all translates of $\bf det$, which is a $L$-submodule of the space $\mathcal P=\mathcal P(V)$.

Let us recall the decomposition of $\mathcal W$  into simple $L$-submodules (see \cite{FG}). 
Fix a Jordan frame $\{c_1,c_2,\dots, c_r\}$, and for each $1\leq k\leq r$, set $e_k = c_1+\dots +c_k$ and let
\[V^{(k)} = V(e_k, 1)= \{ x\in V:  e_k x = x\}.
\]
Further, let $P^{(k)}$ be the orthogonal projection on $V^{(k)}$, and let ${\bf det}^{(k)}$ be the determinant of the euclidean subalgebra $V^{(k)}$. Finally, define the \emph{principal minor} $\Delta_k$ by
\[\Delta_k(x) = {\bf det}^{(k)}\left( P^{(k)} x\right)\ .
\] 

\begin{lemma} The decomposition of $\mathcal W$ into simple $L$-modules is given by
\begin{equation}\mathcal W = \bigoplus_{k=0}^r \ \mathcal W_k
\end{equation}
where $\mathcal W_k=\mathcal P_k\cap \mathcal W$. Moreover, for $1\leq k\leq r$, $\mathcal W_k$ is the $L$-submodule generated by $\Delta_k$ and $\mathcal W_0 = \mathbb R$.

\end{lemma}

The remarkable fact is that the primary decomposition of $\mathcal W$ in homogeneous components coincides with the $L$-decomposition in irreducibles. It has very important consequences.

Let us apply Proposition \ref{propgenTaylor} to the polynomial $\bf det $ and to $\mathcal W$. For each $ 0\leq k\leq r$, let $d_k=\dim \mathcal W_k$ and
choose an orthonormal basis $\big(p_{k,i}\big)_{1\leq i\leq d_k}$ of $\mathcal W_k$ for the Fischer inner product. Then 
\begin{equation}\label{basisdet}
\bigcup_{k=0}^r \big(p_{k,i}\big)_{1\leq i\leq d_k}
\end{equation} 
is an orthonormal basis of $\mathcal W$.

More generally, let $p\in \mathcal W_k$ for some $ 1\leq k\leq r$, and let $\mathcal R(p)$ be the subspace generated by all translates of $p$. As $p$ is homogeneous, $\mathcal W(p)$ splits as
\[\mathcal W(p) = \bigoplus_{\ell=0}^k\  \left(\mathcal W(p)\cap \mathcal P_\ell\right).
\]
As $\mathcal W(p)\cap \mathcal P_\ell\subset \mathcal W\cap \mathcal P_\ell = \mathcal W_\ell$, it is possible to choose an orthonormal basis of $\mathcal W(p)$ of the following form
\begin{equation}\label{basisp}
\bigcup_{\ell = 0}^k \big(q_{\ell, i} \big)_{1\leq i\leq \dim (\mathcal W({p})\cap \mathcal P_\ell)}
\end{equation}
where $q_{\ell,i}\in \mathcal W_\ell$.

\smallskip

{\bf Step 3.} 
We are now ready for the proof of Theorem \ref{Dst}. The last key ingredient which is needed is a consequence (not explicitly stated in \cite{bck1}) of results obtained by Faraut and Kor\'anyi (see Proposition  VII.1.5 and Proposition XI.5.1 in \cite{FK}).

\begin{proposition}
Let $p\in \mathcal W_k.$ We have 
\begin{itemize}
\item[{i)}] The expression
\begin{equation}p^\sharp(x) = p(x^{-1}) {\bf det} (x),
\end{equation}
a priori defined for $x$ invertible, extends to $V$ as a polynomial, which belongs to $\mathcal W_{r-k}$.
\item[{ii)}] For any $\lambda\in \mathbb R$, and for $x\in V$ invertible,
\begin{equation}\label{BSpk}
p\left(\frac{\partial}{\partial x} \right) {\bf det}( x)^{\lambda+1} = \prod_{j=1}^k\left( (\lambda+1+\frac{d}{2} (j-1)\right) p^\sharp(x) {\bf det}( x)^\lambda.
\end{equation} 
\end{itemize}
\end{proposition}
By letting $\lambda = 0$ in \eqref{BSpk}, notice that $p^\sharp$ is proportional to
$\breve{p}= \partial(p){\bf det}$.

To motivate the reader, let us formulate and prove a close but simpler result, which will be used later. 

\begin{proposition}\label{defcst} There exists a polynomial $c_{\lambda,\mu}$ on $V\times V$ such that, for any $x,y\in \Omega$
\begin{equation}
{\bf det}\left(\frac{\partial}{\partial  x} - \frac{\partial}{\partial y} \right) {\bf det}( x)^{s+1} {\bf det}( y)^{t+1} = c_{\lambda,\mu}(x,y){\bf det}( x)^s {\bf det}( y)^t\ .
\end{equation}
The polynomial $c_{\lambda,\mu}$ is homogeneous of degree $r$. Its coefficents depend polynomially on $s,t$.
\end{proposition}
\begin{proof} By Proposition \ref{propgenTaylor} $i)$, and for the basis $\big( p_{k,i}\big)_{0\leq k\leq r, 1\leq i\leq d_k}$ of $\mathcal W$ presented in \eqref{basisdet}
\begin{equation}\label{det(x-y)}
{\bf det}(x-y) = \sum_{k=0}^r (-1)^{r-k} \sum_{i=1}^{d_k}p_{k,i}(x) \breve{p}_{k,i}(y).
\end{equation}
Hence
\begin{multline} {\bf det}\left(\frac{\partial}{\partial  x} - \frac{\partial}{\partial y} \right) {\bf det}( x)^{s+1} {\bf det}( y)^{t+1} \\
= \sum_{k=1}^r (-1)^{r-k} \sum_{i=1}^{d_k} p_{k,i}\left(\frac{\partial}{\partial x}\right) {\bf det}( x)^{s+1}\breve{p}_{k,i}\left( \frac{\partial}{\partial y}\right){\bf det}( y)^{t+1}\ .
\end{multline}
Use \eqref{BSpk} to compute $p_{k,i}\left(\frac{\partial}{\partial x}\right) {\bf det}( x)^{s+1}$ and $\breve{p}_{k,i}\left( \frac{\partial}{\partial y}\right){\bf det}( y)^{t+1}$. The conclusion follows easily.
\end{proof}
The proof of Theorem \ref{Dst} follows similar lines. Let $f$ be a smooth function  on $V\times V$. Using again \eqref{det(x-y)}, we are reduced to calculate, for $0\leq k\leq r$ and $1\leq i\leq d_k,$
\begin{equation}\label{full}
p_{k,i}\left(\frac{\partial}{\partial x}\right)\left( {\bf det}( x)^s\ \left(\breve{p}_{k,i} \left(\frac{\partial}{\partial y}\right) {\bf det}( y)^t f(x,y)\right)\right).
\end{equation}
Compute first the inner expression
\begin{equation}\label{inner}
\breve{p}_{k,i} \left(\frac{\partial}{\partial y}\right) {\bf det}( y)^t f(x,y),
\end{equation}
for which we use Proposition \ref{propgenTaylor} ii), applied to $\breve{p}_{k,i}\in \mathcal W_{r-k}$ and the basis of $\mathcal R(\breve{p}_{k,i})$ presented in \eqref{basisp}. Hence the inner term \eqref{inner} is equal to
\[{\bf det}( y)^{t-1} d_{k,i,t} \left(y,\frac{\partial}{\partial y} \right)f(x,y)
\]
where $d_{k,i,t}$ is a polynomial on $V\times V$. It remains to perform the derivation with respect to the variable $x$, which is treated in a similar manner, showing that  the term \eqref{full} is equal to 
\[{\bf det}( x)^{s-1} {\bf det}( y)^{t-1} d_{k,i,s,t}\left(x,y, \frac{\partial}{\partial x}, \frac{\partial}{\partial y}\right) f(x,y),
\]
where $d_{k,i,s,t}$ is a polynomial on $V\times V\times V\times V$. It remains to sum the terms to finish the proof of Theorem \ref{Dst}.

The results obtained for euclidean Jordan algebras (i.e. type I) are easily extended to other cases, first to type IV by an appropriate complexification of the main identity \eqref{Fourierversion}, and then by elementary tricks to yield the result for type II and type III. We again refer to \cite{bck1} for details.

\begin{example}\label{pq}  Recall from Section 3 the example $V=\mathbb R^{p,q},$  where $${\bf det}(x)=P(x)=x_1^2+\cdots + x_p^2-x_{p+1}^2-\cdots -x_{n}^2.$$
Denote by $P(x,y)$ the corresponding symmetric bilinear form on $V\times V.$ In this example, the differential operator $D_{s,t}$ is given explicitly by 
\begin{equation}\label{DstRpq} 
\begin{array} {rcl}
D_{s,t} &=& \displaystyle  P(x)P(y)\,P \left(\frac{\partial}{\partial x} -\frac{\partial}{\partial y}\right)\\
&&\displaystyle  +4s\, P(y) \sum_{i=1}^n x_j\left( \frac{\partial }{\partial x_j}-  \frac{\partial }{\partial y_j}\right)+4t\,P(x)\sum_{j=1}^n y_j\left( \frac{\partial }{\partial y_j}-  \frac{\partial }{\partial x_j}\right)\\
&&\displaystyle  +2t(2t-2+n) P(x)-8stP(x,y) +2s(2s-2+n) P(y).
\end{array}
\end{equation}
\end{example}

Let us state  the (slightly more precise) version of Theorem \ref{Flambdamu} in the noncompact picture.

\begin{theorem} The differential operators $F_{\lambda, \mu}$ have polynomials coefficients on $V\times V$, and they depend polynomially on $(\lambda,\mu)\in \mathbb C\times \mathbb C$.  For any $g\in G$,
\begin{equation}
F_{\lambda,\mu}\circ \big(\pi_{\lambda,\varepsilon}(g)\otimes (\pi_{\mu, \eta} \big)(g)=\big(\pi_{\lambda+1,-\varepsilon}(g)\otimes \pi_{\mu+1,-\eta}(g)\big)\circ F_{\lambda,\mu}\\ .
\end{equation} 
\end{theorem}

\section{Construction of the generalized Rankin-Cohen operators}
Let $\operatorname{res} : \mathcal{S}(V \times V) \longrightarrow \mathcal{S}(V)$ be the restriction operator from $V \times V$ to $\operatorname{diag}(V)=\{(x, x), x \in V\} \simeq V$ given by
$$
\operatorname{res}(f)(x)=f(x, x), \quad \text { for } x \in V .
$$
One can easily prove that  for $(\lambda, \varepsilon),(\mu, \eta) \in \mathbb{C} \times\{\pm\}$, and for any $g \in G$
\begin{equation}
\operatorname{res} \circ\left(\pi_{\lambda, \varepsilon}(g) \otimes \pi_{\mu, \eta}(g)\right )=\pi_{\lambda+\mu, \varepsilon \eta}(g) \circ \text {res }
\end{equation}

For any positive integer $k$, let
$$
F_{\lambda, \mu}^{(k)}=F_{\lambda+k-1, \mu+k-1} \circ \cdots \circ F_{\lambda, \mu}
$$
and
$$
B_{\lambda, \mu}^{(k)}=\operatorname{res} \circ F_{\lambda, \mu}^{(k)}
$$

The covariance property of the operators $F_{\lambda, \mu}$ and of res imply the following statement.

\begin{proposition}
For all $(\lambda, \varepsilon)$ and $(\mu, \eta)$ in $\mathbb{C} \times\{\pm\}$ and for all $k$ in $\mathbb{N}^*$, the operators $B_{\lambda, \mu}^{(k)}$ are covariant with respect to 
$\left(\pi_{\lambda, \varepsilon} \otimes \pi_{\mu, \eta}, \pi_{\lambda+\mu+2 k, \varepsilon \eta}\right)$, i.e. for any $g\in G$,
\begin{equation}
B_{\lambda, \mu}^{(k)} \circ\left(\pi_{\lambda, \varepsilon}(g) \otimes \pi_{\mu, \eta}(g)\right)=\pi_{\lambda+\mu+2 k, \varepsilon \eta}(g) \circ B_{\lambda, \mu}^{(k)}.
\end{equation}

 \end{proposition}

 The operators $B_{\lambda, \mu}^{(k)}$ generalize the classical Rankin-Cohen operators, or more precisely their real valued version on $\mathbb R\times \mathbb R$.  
 \begin{example}  Assume that  $V=\mathbb R^{p,q}$. 
Using the same notation as in Example \ref{pq}, the source operator is given by 
$$
\begin{array}{rl}
\displaystyle  F_{\lambda,  \mu  }
 &=  \displaystyle-P(x-y) P\left(\frac{\partial}{\partial x}\right) P\left(\frac{\partial}{\partial y}\right) 
\\
& \displaystyle\quad+4(-\lambda+\frac{n}{2}-1) \sum_{j=1}^n (x_j-y_j)\frac{\partial}{\partial x_j}P\left(\frac{\partial}{\partial y}\right)  + 4(-\mu+\frac{n}{2}-1) \sum_{j=1}^n (y_j-x_j)\frac{\partial}{\partial x_j}P\left(\frac{\partial}{\partial x}\right)
\\
&\displaystyle\quad 
+4\lambda (-\lambda+\frac{n}{2}-1)  P\left(\frac{\partial}{\partial y} \right)
+4\mu (-\mu+\frac{n}{2}-1)  P\left(\frac{\partial}{\partial x} \right)
\\
& \displaystyle\quad+8(-\lambda+\frac{n}{2}-1)(-\mu+\frac{n}{2}-1) P\left(\frac{\partial}{\partial x}, \frac{\partial}{\partial y}\right).
\end{array}$$
In particular, 
\begin{multline*}
B^{(1)}_{\lambda, \mu}=4\, \operatorname{res}\Big\{
\mu (-\mu+\frac{n}{2}-1) P\left(\frac{ \partial}{\partial x}\right) +\lambda (-\lambda+\frac{n}{2}-1)  P\left( \frac{ \partial}{\partial y}\right) \\
+2(-\lambda+\frac{n}{2}-1)(-\mu+\frac{n}{2}-1) P\left(\frac{ \partial}{\partial x}, \frac{ \partial}{\partial y}\right)\Big\}.
\end{multline*}
See \cite[Section 9]{bck1} for more details.

\end{example}

To get a more explicit expression of these operators, we need to extend to bi-differential operators the classical notion of symbols. For the present situation, it will be enough to consider differential operators with polynomials coefficients.
Let $B\colon C^{\infty}(V \times V) \longrightarrow C^{\infty}(V)$  be a bi-differential operator with polynomial coefficients. Define its symbol $b(x,\xi,\eta)$ for $x\in V, (\xi,\eta) \in V\times V$ by the identity
$$
B\left(e^{i(x  \xi+y  \zeta)}\right)_{x=y}=b(x, \xi, \zeta) e^{i(x  \xi+x  \zeta)}
$$
If $B$ is a bi-differential operator with constant coefficients, its symbol does not depend on $x$ and is simply denoted by $b(\xi,\zeta)$.

By repeated applications of Proposition \ref{defcst}, for each $k\in \mathbb N$, there exists a polynomial $c_{s,t}^{(k)}(\xi, \zeta)$ such that
\begin{equation}\label{Rodrigues}
 {\bf det} \left(\frac{\partial}{\partial  \xi} - \frac{\partial}{\partial \zeta} \right)^k {\bf det}( \xi)^{\lambda+k}  {\bf det} ( \zeta)^{\mu+k} = c^{(k)}_{\lambda,\mu}( \xi,\zeta){\bf det} ( \xi)^\lambda {\bf det}( \zeta)^\mu\ .
\end{equation}

\begin{theorem} The symbol of the bi-differential operator $B_{\lambda, \mu}^{(k)}$ is given by
\begin{equation}
b_{\lambda,\mu}^{(k)}(\xi, \zeta) = c_{\lambda-\frac{n}{r},\, \mu-\frac{n}{r}}^{(k)}(\xi, \zeta)\ .
\end{equation}
\end{theorem}
\begin{proof}
It is possible to compose a differential operator $F$ on $V\times V$ followed by a bi-differential operator $B$ on $V\times V$, yielding a bi-differential operator $B\circ F$ on $V\times V$. There is a corresponding composition formula expressing the symbol $b(x,\xi,\eta)\# f(x,y,\xi,\zeta)$ of $B\circ F$, in terms of the symbols of $B$ and $F$, similar to the classical composition formula for symbols of differential operators. Now observe that
\begin{eqnarray*}
B_{\lambda, \mu}^{(k)}&=&\operatorname{res} \circ\left(F_{(\lambda+k-1,\, \mu+k-1)} \circ \cdots \circ F_{\lambda+1,\, \mu+1}\right) \circ F_{\lambda,\, \mu} \\
&=&\operatorname{res} \circ F_{\lambda+1,\, \mu+1}^{(k-1)} \circ F_{\lambda, \,\mu}\\
&=& B_{\lambda+1,\mu+1}^{(k-1)} \circ F_{\lambda,\mu}
\end{eqnarray*}
which implies in terms of symbols
\begin{eqnarray*}
b_{\lambda,\, \mu}^{(k)}(\xi, \zeta) &=&\left(b_{\lambda+1, \mu+1}^{(k-1)} \# f_{\lambda, \,\mu}\right)(\xi, \zeta)\\
&=&D_{\lambda-\frac{n}{r}+1,\,  \mu-\frac{n}{r}+1}\left(b_{\lambda+1,\, \mu+1}^{(k-1)}\right)(\xi, \zeta),
\end{eqnarray*}
see the proof of Proposition 3.4 in \cite{Clerc-Rod} for the last equality. Hence
$$
\begin{aligned}
&\  {\bf det}(  \xi)^{\lambda-\frac{n}{r}}  {\bf det}( \zeta)^{\mu-\frac{n}{r}} \,b^{(k)}_{\lambda, \mu}(\xi, \zeta)\\
&={\bf det}( \xi)^{\lambda-\frac{n}{r}} {\bf det}( \zeta)^{\mu-\frac{n}{r}} \,\big(D_{\lambda-\frac{n}{r}+1,\, \mu-\frac{n}{r}+1}\, b_{\lambda+1,\, \mu+1}^{(k-1)}\big)(\xi, \eta) \\
&={\bf det} \left(\frac{\partial}{\partial \xi}-\frac{\partial}{\partial \zeta}\right) \left({\bf det}( \xi)^{\lambda-\frac{n}{r}+1} {\bf det}( \zeta)^{\mu-\frac{n}{r}+1}b_{\lambda+1, \mu+1}^{(k-1)}(\xi, \zeta)\right)  ,
\end{aligned}
$$
and this is how we come up with the following result.

\begin{lemma}
 The symbols $b_{\lambda, \mu}^{(k)}$ satisfy the following recurrence relation
\begin{equation}
\begin{aligned}
&{\bf det}(  \xi)^{\lambda-\frac{n}{r}}  {\bf det}( \zeta)^{\mu-\frac{n}{r}} b_{\lambda, \mu}^{(k)}(\xi, \zeta)\\
&={\bf det} \left(\frac{\partial}{\partial \xi}-\frac{\partial}{\partial \zeta}\right)
\left({\bf det}( \xi)^{\lambda-\frac{n}{r}+1} {\bf det}(\zeta)^{\mu-\frac{n}{r}+1} \,b_{\lambda+1, \,\mu+1}^{(k-1)}(\xi, \zeta) \right)
\end{aligned}
\end{equation}
\end{lemma}

To conclude, it remains to calculate the solution of the recurrence relation. 
Now the family $c^{(k)}_{\lambda, \mu}$ satisfies the recurrence relation
\begin{eqnarray*}
&&c_{\lambda, \mu}(\xi,\zeta) {\bf det}(\xi)^\lambda {\bf det} (\zeta)^\mu \\
&=&{\bf det} \left(\frac{\partial}{\partial \xi}- \frac{\partial}{\partial \zeta}\right)\left ({\bf det} \left(\frac{\partial}{\partial \xi}- \frac{\partial}{\partial \zeta}\right)^{k-1}
{\bf det}( \xi)^{(\lambda+1)+(k-1)} {\bf det}( \zeta)^{(\mu+1)+(k-1)}\right)\\
 &=&{\bf det} \left(\frac{\partial}{\partial  \xi} - \frac{\partial}{\partial \zeta} \right) \left(c_{\lambda+1,\mu+1}^{(k-1)}(\xi, \zeta){\bf det}( \xi)^{\lambda+1}{\bf det}( \zeta)^{\mu+1}\right)\ .
\end{eqnarray*}
Hence the family $c^{(k)}_{\lambda-\frac{n}{r},\, \mu-\frac{n}{r}}$ satisfies the same recurrence relation as the family $b_{\lambda,\, \mu}^{(k)}$ and are both equal to $1$ for $k=0$ and hence, they coincide for every $k\in \mathbb N$. 
\end{proof}

\begin{remark}  The  definition of the polynomials $c_{\lambda, \mu}^{(k)}$ reminds of the definition of the Jacobi polynomials using the \emph{Rodrigues formula}, see \cite{c23} Section 5 for more details.

\end{remark}

\section{Other applications}
Below we shall  assure that our new approach for constructing symmetry breaking differential operators   is efficient to produce new examples of such operators in many different geometric contexts. \\
 
 \noindent 
 {\bf The case of differential forms.} 
In  \cite{bck2} we were able to built bi-differential operators acting on spaces of differential forms which are covariant for the Lie group $G={\rm SO}_0(1,n+1),$ the conformal group of $\mathbb R^n. $ Below we will briefly summarize the main results in \cite{bck2}. All the detailed proofs can be found in [loc.cit.].

Let  $ \mathbb R^{1,n+1}$  be the $n+2$-dimensional real vector space equipped with the Lorentzian quadratic form $$[{\bf x},{\bf x}] = x_0^2-(x_1^2+\dots+x_{n+1}^2),\qquad {\bf x}=(x_0,x_1,\ldots, x_{n+1}).$$ Let $G={\rm SO}_0(1,n+1)$ be the connected component of the identity in the group of isometries for the Lorentzian form on $\mathbb R^{1,n+1}$.   The action of $G$ on the unit sphere $S^n$  of $\mathbb R^{n+1}$ is given as follows: For $x'=(x'_1, \ldots, x'_{n+1})\in S^n$ and $g\in G,$ observe that the component $(g(1,x'))_0>0$, which allows to define  the element $g(x')\in S^n$ by  $$(1, g(x'))= (g(1,x'))_0^{-1}\, g(1,x').$$  By  the inverse map of the stereographic projection, the action of $G$ on $S^n$ can be transferred to a  rational   action (not everywhere defined) on $\mathbb R^n$, for which we still use the notation $G\times \mathbb R^n\ni(g,x) \longmapsto g(x)\in \mathbb R^n$.

Below we shall use the following standard notation for the Iwasawa decomposition of $G,$ 
$$G\ni g = \overline n(g) m(g) a_{t(g)} n(g)\in \overline N \times M\times A \times N.$$ We will identify the elements $  n_x\in N$    with $x\in \mathbb R^n,$ and,  for $\lambda\in \mathbb C,$ we write  $a_{t(g)}^\lambda=e^{\lambda t(g)}.$ 

 For $0\leq k\leq n,$ we denote by $\Lambda^k$ the vector space of complex-valued  alternating multilinear $k$-forms on $\mathbb R^n.$  Let $\mathcal E^k(\mathbb R^n)=C^\infty(\mathbb R^n, \Lambda^k)$ be the space of smooth complex-valued differential $k$-forms on   $\mathbb R^n$.   For  $g\in G$ and $\omega \in \mathcal E^k(\mathbb R^n)$ let 
\begin{equation}\label{spin}  \pi_\lambda^k(g) \omega (x) = e^{-(\lambda+k)t(g^{-1}\overline n_x)}\sigma_k \big(m(g^{-1}\overline n_x)\big)^{-1} \omega(g^{-1}(x) ),\end{equation}
where $\sigma_k$ is the   representation of $M={\rm SO}(n)$ on  $\Lambda^k.$ That is,
\begin{equation}\label{piPS} 
  \pi_\lambda^k\simeq {\rm Ind}_{MAN}^G (\sigma_k\otimes \chi_{\lambda+k}\otimes 1),
\end{equation}
 where, for $\lambda\in\mathbb{C}$, we denote by $\chi_\lambda$ the character of $A$  given by $\chi_\lambda(a_t)=e^{\lambda t}$.

In the present framework, the {Knapp-Stein intertwining  operators} take the form $$ I^k_\lambda \omega( x) = \int_{\mathbb R^n} \mathcal R_{-2n+2\lambda}^k(x-y)\, \omega(y) dy,\qquad \omega\in \mathcal E^k(\mathbb R^n)$$ 
where $\mathcal R_{-2n+2\lambda}^k$ is the tempered distribution defined by 
  \begin{equation}\label{riezf2} 
  \langle \mathcal R_{s}^k,  \omega\rangle    =   \pi^{-{n\over 2}}  2^{-{s}-n+1} {{\Gamma\left(-{{s}\over 2}+1\right)}\over {\Gamma\left({{{s}+n}\over 2}\right)}}      \int_{\mathbb R^n}    \Vert x\Vert ^{s-2} ({\boldsymbol \iota}_{x} {\boldsymbol \varepsilon}_{x} -{\boldsymbol \varepsilon}_{x} {\boldsymbol \iota}_{x}) \omega(x) dx,
  \end{equation} 
  with ${\boldsymbol \iota}_{x}$ and ${\boldsymbol \varepsilon}_{x}$ are, respectively,  the interior and the exterior products of $k$-forms.
 For $-n<\Rel {s} <0$,  $ \mathcal R_{s}^k$  defines a tempered distribution depending holomorphically on $s.$ By \cite[\S 3.2]{fo},    $  \mathcal R_{s}^k$ can be analytically  continued to $\mathbb C$, giving a meromorphic family of tempered distributions.  When  $k=0,$ the identity 
  ${\boldsymbol \iota}_{x} {\boldsymbol \varepsilon}_{x} +{\boldsymbol \varepsilon}_{x} {\boldsymbol \iota}_{x}=\Vert x\Vert^ 2\Id_\mathcal E  $ implies immediately that $\mathcal R_{s}^0$ is nothing but the   classical Riesz distribution, up to the multiplication by $-s$.   In particular, we have
  \begin{equation*}\label{covI}
  I^k_\lambda \circ   \pi^k_\lambda(g) =  \pi_{n-\lambda}^k(g)\circ   I^k_\lambda\ .
\end{equation*}

  For $0\leq k,\ell\leq n,$  let $\mathcal E^{k,\ell}(\mathbb R^n\times \mathbb R^n)=C^\infty (\mathbb R^n\times \mathbb R^n, \Lambda^k\otimes \Lambda^\ell)  $ be the space of differential forms of bidegree $(k,\ell).$  The source operator $F^{k,\ell}_{\lambda, \mu}$ defined on   $\mathcal E^{k,\ell}(\mathbb R^n\times \mathbb R^n)$ and satisfies  
\begin{equation*} 
F^{k,\ell}_{\lambda, \mu} \circ \big(   \pi^{k}_\lambda \otimes   \pi^{\ell}_\mu\big)(g)  \omega = \big(  \pi_{\lambda+1}^{k} \otimes   \pi_{\mu+1}^{\ell}\big)(g)\circ F_{\lambda, \mu}^{k,\ell}\,\omega
\end{equation*}
 is  given explicitly  by 
\begin{eqnarray*}\label{F-lambda-mu-kl2}
  F_{\lambda,\mu}^{k,\ell}&=&    16\Bigl\{\Vert x-y\Vert^2        \widetilde \boxvoid_{k,\lambda}\otimes  \widetilde \boxvoid_{\ell,\mu} \nonumber\\ 
  &&- 2 \sum_{j=1}^n (x_j-y_j)   \left( (2\lambda-n+2)  \widetilde \nabla_{k,\lambda,j}   \otimes  \widetilde \boxvoid_{\ell,\mu}
  -    (2\mu-n+2)  \widetilde \boxvoid_{k,\lambda} \otimes     \widetilde\nabla_{\ell,\mu,j}  \right) \nonumber\\ 
  &&- 2 (2\lambda-n+2)(2\mu-n+2)  \sum_{j=1}^n     \widetilde \nabla_{k,\lambda,j}   \otimes      \widetilde \nabla_{\ell,\mu,j}\nonumber \\
  &&-2   (2\mu-n+2)(\mu+1)(\mu-\ell)(\mu-n+\ell)     \widetilde \boxvoid_{k,\lambda} \otimes  \Id_{{\mathcal E}^\ell}\nonumber\\
&&  -2    (2\lambda-n+2)(\lambda+1)(\lambda-k)(\lambda-n+k)      \Id_{{\mathcal E}^k} \otimes  \widetilde \boxvoid_{\ell,\mu}\Bigr\},\nonumber
 \end{eqnarray*}
 where
 \begin{eqnarray*}
   \widetilde \boxvoid_{p,q}&=&(q-n+p)(q-p+1) \cdd\dd+(q-n+p+1) (q-p) \dd\cdd \\
  \widetilde  \nabla_{p,q,j} &=&(q-n+p)(q-p+1)\partial_{x_j}-(q-p)\boldsymbol \varepsilon_{e_j}\cdd-(q-n+p)\cdd \boldsymbol \iota_{e_j} .
 \end{eqnarray*}
 Here $\dd$ is the exterior differential    ${\boldsymbol d}: \mathcal E^k(\mathbb R^n)   \longrightarrow  \mathcal E^{k+1} (\mathbb R^n)$  defined   by
 ${\boldsymbol d}    =\sum_{j=1}^n {\boldsymbol \varepsilon}_{e_j} \partial_{e_j},$ while ${\boldsymbol   \delta} $ 
is  the  co-differential  ${\boldsymbol   \delta} :    \mathcal E^{k+1}  (\mathbb R^n) \longrightarrow  \mathcal E^k(\mathbb R^n) $ given  by 
 ${\boldsymbol   \delta}  = - \sum_{j=1}^n   {\boldsymbol \iota}_{e_j}      \partial_{e_j} .$
More generally, for any integer $m\geq 1$, the operator 
$$ F^{k,\ell}_{\lambda, \mu;m} := F^{k,\ell}_{\lambda+m-1, \mu+m-1}\circ \dots \circ F^{k,\ell}_{\lambda+1, \mu+1}\circ F^{k,\ell}_{\lambda, \mu}$$
 intertwines the representations $  \pi^{k}_\lambda\otimes   \pi^{\ell}_\mu$ and $  \pi^{k}_{\lambda+m}\otimes   \pi^{\ell}_{\mu+m}$.

Let $\res : \mathcal E^{k,\ell} ( \mathbb R^n\times \mathbb R^n)\longrightarrow  C^\infty ( \mathbb R^n,  \Lambda^k\otimes \Lambda^\ell   ) $ be the restriction map defined by
$$(\res \omega) (x) =  \omega(x,x).$$
 The map $\res$ intertwines the representations $\pi_\lambda^{k}\otimes  \pi_\mu^{\ell}$ and $\Ind_{MAN}^G \big((\sigma_k\otimes \sigma_\ell)\otimes \chi_{\lambda+\mu+k+\ell}\otimes 1\big)$. However, as a representation of $M ={\rm SO}(n)$, the tensor product  $\sigma_k\otimes \sigma_\ell$ is in general not irreducible. Let $\Gamma $ be a minimal invariant subspace of 
$\Lambda^k\otimes \Lambda^\ell $ under the action of ${\rm SO}(n)$. Let $\sigma_\Gamma$ be the corresponding irreducible representation of ${\rm SO}(n)$ on $\Gamma, $ and let
$p_\Gamma$ be the orthogonal projection on $\Gamma $.   Then, the map $\res_\Gamma:=p_\Gamma \circ\res$ intertwines the representations $\pi_\lambda^{k}\otimes  \pi_\mu^{\ell}$ and $\Ind_P^G \big(\sigma_\Gamma\otimes \chi_{\lambda+\mu+k+\ell}\otimes 1\big)$. 

Finally, from putting all pieces together,  the  operator  $ B_{\lambda, \mu;m}^{k,\ell;\Gamma}: \mathcal E^{k,\ell} ( \mathbb R^n\times \mathbb R^n)\longrightarrow   
C^\infty(\mathbb R^n, \Gamma )$  defined  by
$$
B_{\lambda, \mu;m}^{k,\ell;\Gamma} := \res_\Gamma \circ\, F_{\lambda, \mu;m}^{k,\ell}\,,
$$ is a bi-differential operator and covariant with respect to  $\pi_\lambda^{k}\otimes  \pi_\mu^{\ell}$ and $\Ind_P^G (\sigma_\Gamma\otimes \chi_{\lambda+\mu+k+\ell+2m}\otimes 1)$.

In some cases it is possible to give an explicit  expression for $B_{\lambda, \mu;m}^{k,\ell;\Gamma}.$ For instance,  if  $0\leq k+\ell\leq n$, then the representation $\Lambda^{k+\ell}$ appears in the decomposition of the tensor product $\Lambda^k \otimes \Lambda^\ell $ with multiplicity one, and the projection (up to a normalization factor) is given by
$$p_{\Lambda^{k+\ell}}(\omega\otimes \eta) = \omega\wedge \eta.$$
Thus, for $m=1$, the bi-differential operator $B_{\lambda, \mu;1}^{k,\ell;\Lambda^{k+\ell}}$ is given by
 \begin{eqnarray*}
  B_{\lambda, \mu;1}^{k,\ell;\Lambda^{k+\ell}}(\omega\otimes \eta) (x) &=&     
  - 32 \Bigl\{ (2\mu-n+2)(\mu+1)(\mu-\ell)(\mu-n+\ell)     \widetilde \boxvoid_{k,\lambda}\omega(x)\wedge\eta(x)    \nonumber\\
  &&+(2\lambda-n+2)(2\mu-n+2)  \sum_{j=1}^n     \widetilde \nabla_{k,\lambda,j}\omega(x)\wedge       \widetilde \nabla_{\ell,\mu,j}\eta(x)  \nonumber\\ 
&&  +  (2\lambda-n+2)(\lambda+1)(\lambda-k)(\lambda-n+k)     \omega(x) \wedge \widetilde \boxvoid_{\ell,\mu}\eta(x) \Big\}
 \end{eqnarray*}
 If in addition $k=\ell=0,$ i.e. $\omega,\eta \in C^\infty(\mathbb R^n),$ then 
\begin{multline*}
\shoveright{B_{\lambda, \mu;1}^{0,0;\mathbb C} (\omega\otimes \eta) (x) =    
-64(\lambda+1)(\lambda-n)(\mu+1)(\mu-n)  \Big\{\mu (\mu-\frac{n}{2}+1)     \left(\Delta\left(\partial_x\right) \omega\right)(x) \eta(x)}\\
\shoveright{+2(\lambda-\frac{n}{2}+1) (\mu-\frac{n}{2}+1) \sum_{j=1}^n \left(\partial_{x_j}\omega\right)(x)
 \left(\partial_{x_j}\eta\right)(x) }\\
+  \lambda (\lambda-\frac{n}{2}+1)    \omega(x)  \left(\Delta\left(\partial_x\right) \eta\right)(x)    \Big\},
\end{multline*}
which coincides, up to a scalar factor,  with the  multidimensional Rankin-Cohen operators   in Section 5. \\

\noindent
 {\bf The case of spinors.} The construction of covariant bi-differential operators acting on the spinor bundle has been developed in \cite{ck}. We will give below a brief  overview of the main results. For more details   see \cite{ck}.

The  conformal spin group ${G}=\operatorname{Spin}_0(1, n+1)$, a double cover of the conformal group $G=\mathrm{SO}_0(1, n+1)$,   acts by conformal rational transformations on $\mathbb{R}^n$. This action gives rise to   the principal series representation  $\pi_{\rho , \lambda}=\mathrm{Ind}_{  P}^{  G}(\rho\otimes\chi_\lambda\otimes 1)$,  parameterized by a spinor representation $\left(\rho,\mathbb S\right)$ of $ {M}=\operatorname{Spin}(n)$ and a character $\chi_\lambda$ of $A$ for $\lambda\in\mathfrak a_{\mathbb C}\simeq\mathbb C$. Above $ P=  M A N$ is a maximal parabolic subgroup of $ G$ where $A\simeq\mathbb R$, $N\simeq\mathbb R^n$ and $  M=\operatorname{Spin}(n)$ is  the centralizer of $A$ in the maximal compact subgroup $  K=\operatorname{Spin}(n+1)$ of $ G$. Let $ M'$ be the normalizer of $A$ in $ K$. Then the Weyl group $  M'/  M$ has two elements. As a representative of the non-trivial Weyl group element choose $w= e_1e_{n+1}$, where  $\{e_j\}_{j=0}^{n+1}$ is the standard basis of $\mathbb R^{1,n+1}$.

The non-compact realization of the representation $\pi_{\rho, \lambda}$, which will be used,  turns out to be more  appropriate for Clifford algebra calculus. It is given by
\[\pi_{\rho,\lambda} (g) f (\overline n) = {\chi_\lambda\big(a(g^{-1}\overline n)\big)^{-1}} \rho\big(m(g^{-1}\overline n)\big)^{-1} f\left(g^{-1}(\overline n)\right)\!,
\]
where $f$ is a smooth $\mathbb S$-valued  function on $\overline N\simeq\mathbb R^n$.

The Knapp-Stein operators  $J_{ \rho,\lambda}$
are intertwining operators between $\pi_{ \rho, \lambda}$ and $\pi_{w \rho, n-\lambda}$.  However,  $w \rho$ and $ \rho$ equivalent (namely $ \rho(-e_1)\circ w \rho(m) =   \rho(m)\circ  \rho(-e_1) $), thus we let
 \[I_{ \rho, \lambda} = \rho(-e_1)\circ J_{ \rho, \lambda} . \]
which in fact are intertwining operators between $\pi_{\rho,\lambda}$ and $\pi_{\rho,n-\lambda}$. 

Now consider simultaneously  a spinor representation $(\rho, \mathbb S)$  and its dual $(\rho', \mathbb S')$. Consider for $\lambda, \mu \in \mathbb C$, the corresponding induced representations are $\pi_{\rho,\lambda}$ and $\pi_{\rho',\mu}$.   
The tensor product $\pi_{\rho,\lambda} \otimes \pi_{\rho',\mu}$ is a representation of ${G}$ by the diagonal action. This representation is not irreducible and one is interested in constructing  symmetry breaking differential operators    $\pi_{\rho ; \lambda} \otimes \pi_{\rho^{\prime} ; \mu} \rightarrow \pi$ into irreducible representations $\pi$ of ${G}$, for instance into another principal series representation $\pi=\pi_{\tau^* , \nu}$ where   $\tau_k^*$ is an exterior power representation of $ M=\mathrm{Spin}(n)$ on $\Lambda^*_k\left(\mathbb{C}^n\right)$.  


 To this aim, we consider for $s, t\in\mathbb C$,   the following  {\it Clifford-Riesz operators} $\fsl{\mathcal R}_s$ and $\fsl{\mathcal R}'_t$  given by
$$
\fsl{\mathcal R}_s f=\int_{\mathbb R^n} \fsl{r}_s(x-y) f(y)dy,\qquad
\fsl{\mathcal R}'_t f=\int_{\mathbb R^n} \fsl{r}'_t(x-y) f(y)dy$$
where 
$\fsl{r}$ and $\fsl{r}'$ are the {\it Clifford-Riesz distributions} 
\begin{equation}\label{CliffordRiesz}
\fsl{r}_s(x) = \vert x\vert^s \rho\left(\frac{x}{\vert x\vert}\right),\qquad
\fsl{r}'_t(x) = \vert x\vert^t \rho'\left(\frac{x}{\vert x\vert}\right),\;
\end{equation}
associated  respectively to the Clifford modules $(\mathbb S,\rho)$ and $(\mathbb S',\rho')$.    

The main observation in this framework is that up to a shift in the parameters, and up to a constant multiple, the Knapp-Stein operators  for  $\pi_{\rho,\lambda}$ and $\pi_{\rho',\mu}$ are essentially the Clifford-Riesz operators. 
More precisely,  
 \[ I_{\rho,\lambda} = \fsl{\mathcal R}_{2\lambda-2n},\qquad  I_{\rho',\mu} = \fsl{\mathcal R}'_{2\mu-2n}.
 \]
Further, let 
 $\mathcal M$ to be the operator on $C^\infty(\mathbb R^n\times \mathbb R^n, \mathbb S\otimes\mathbb S')$ defined  for $f$ a smooth function on $\mathbb R^n\times \mathbb R^n$ with  values in $\mathbb S\otimes \mathbb S'$ by
\begin{equation*}
 {M} f(x,y) = \vert x-y\vert^2 f(x,y) .
\end{equation*}
The operator $ M$ intertwines $\pi_{\rho,\lambda}\otimes\pi_{\rho',\mu}$ and $\pi_{\rho,\lambda-1}\otimes\pi_{\rho',\mu-1}$.\\
Applying in this setting the general construction as in section 5, the source operator $F_{\lambda,\mu}$,   which satisfies $F_{\lambda, \mu} \circ\left(\pi_{\rho,\lambda}(g) \otimes \pi_{\rho',\mu}(g)\right)=\left(\pi_{\rho,\lambda+1}(g) \otimes \pi_{\rho',\mu+1}(g)\right) \circ F_{\lambda, \mu}$, is given by 
\begin{eqnarray}\label{Elambdamu}
F_{\lambda,\mu} &=& \vert x-y\vert^2 \Delta_x\otimes \Delta_y\nonumber\\
&+&2(2\lambda-n+1) \sum_{j=1}^n (x_j-y_j) \frac{\partial}{\partial x_j} \otimes \Delta_y+ 2(2\mu-n+1) \sum_{j=1}^n (y_j-x_j) \Delta_x\otimes \frac{\partial}{\partial y_j}\nonumber\\
&-&2\rho(x-y)\fsl{D}_x\otimes \Delta_y\ -\ 2 \Delta_x\otimes \rho'(x-y) \fsl{D}_y\nonumber\\
\nonumber\\
&+&(2\mu-n+1)(2\mu+1)  \Delta_x\otimes \id\ +\ (2\lambda-n+1) (2\lambda+1)\id\otimes\Delta_y\\
 &-&2(2\lambda-n+1)(2\mu-n+1) \sum_{j=1}^n \frac{\partial}{\partial x_j} \otimes \frac{\partial}{\partial y_j}\nonumber\\
&+&2(2\lambda-n+1) \sum_{j=1}^n \frac{\partial}{\partial x_j}\otimes \rho'(e_j)\fsl{D}'_y+2(2\mu-n+1)\sum_{j=1}^n \rho(e_j)\fsl{D}_x\otimes \frac{\partial}{\partial y_j}\nonumber \\
&-& 2\big(\sum_{j=1}^n \rho(e_j)\fsl{D}_x\otimes  \rho'(e_j)\fsl{D}'_y\big).\nonumber
\end{eqnarray}
Above, $\fsl{D}$ and $\fsl{D}'$ are respectively the Dirac operators associated to the Clifford modules $(\rho,\mathbb S)$ and $(\rho',\mathbb S')$

Let for $1\leq k\leq n$,
$$
\widetilde{\Psi}^{(k)}: C^{\infty}\left(\mathbb{R}^n \times \mathbb{R}^n, \mathbb{S} \otimes \mathbb{S}^{\prime}\right) \rightarrow C^{\infty}\left(\mathbb{R}^n, \Lambda_k^*\left(\mathbb{R}^n\right) \otimes \mathbb{C}\right), \quad f \mapsto\left(\Psi^{(k)} f(x, y)\right)_{\mid x=y}
$$
where
  $\Psi^{(k)}: \mathbb{S} \otimes \mathbb{S}^{\prime} \rightarrow \Lambda_k^*\left(\mathbb{R}^n\right) \otimes \mathbb{C}$ is the usual projection. Then one can prove that $\widetilde{\Psi}^{(k)}$ intertwines 
  $\pi_{\rho, \lambda}\otimes\pi_{\rho',\mu}$ and $\pi_{\tau_k^*,\lambda+\mu}$.

For $m \in \mathbb{N}$, define the operator $F_{\lambda, \mu}^{(m)}: C^{\infty}\left(\mathbb{R}^n \times \mathbb{R}^n, \mathbb{S} \otimes \mathbb{S}^{\prime}\right) \rightarrow C^{\infty}\left(\mathbb{R}^n \times \mathbb{R}^n, \mathbb{S} \otimes \mathbb{S}^{\prime}\right)$ by
$$
F_{\lambda, \mu}^{(m)}=F_{\lambda+m-1, \mu+m-1} \circ \cdots \circ F_{\lambda, \mu} .
$$
It is an intertwining operator of $\pi_{\rho,\lambda} \otimes \pi_{\rho',\mu}$ and $\pi_{\rho,\lambda+m} \otimes \pi_{\rho',\mu+m}$.

Finlay, define  
$$
B_{k ; \lambda, \mu}^{(m)}=\widetilde{\Psi}^{(k)} \circ F_{\lambda, \mu}^{(m)} .
$$
Then the operators $B_{k ; \lambda, \mu}^{(m)}: C^{\infty}\left(\mathbb{R}^n \times \mathbb{R}^n, \mathbb{S} \otimes \mathbb{S}^{\prime}\right) \rightarrow C^{\infty}\left(\mathbb{R}^n, \Lambda_k^*\left(\mathbb{R}^n\right) \otimes \mathbb{C}\right)$ are constant coefficient bi-differential operators and homogeneous of degree $2 m$, and for any $g \in  {G}$, 
$$
B_{k ; \lambda, \mu}^{(m)} \circ\left(\pi_{\rho,\lambda}(g) \otimes \pi_{\rho',\mu}(g)\right)=\pi_{\tau^*_k ; \lambda+\mu+2 m}(g) \circ B_{k ; \lambda, \mu}^{(m)} .
$$

For  $k=0$ and $m=1$,  
the operator
$$
B_{0 ; \lambda, \mu}^{(1)}: C^{\infty}\left(\mathbb{R}^n \times \mathbb{R}^n, \mathbb{S} \otimes \mathbb{S}^{\prime}\right) \rightarrow C^{\infty}\left(\mathbb{R}^n\right)
$$
is given by
$$
\begin{aligned}
B_{0 ; \lambda, \mu}^{(1)}\left(v(\cdot) \otimes w^{\prime}(\cdot)\right)(x) & =(2 \mu-n+1)(2 \mu+1)\left(\Delta v(x), w^{\prime}(x)\right) \\
& +(2 \lambda-n+1)(2 \lambda+1)\left(v(x), \Delta w^{\prime}(x)\right) \\
& -2(2 \lambda-n+1)(2 \mu-n+1) \\
& +\sum_{j=1}^n\left(\frac{\partial}{\partial x_j} v(x), \frac{\partial}{\partial y_j} w^{\prime}(x)\right) \\
& -2(2 \lambda+2 \mu-n+2)\left(\fsl{D} v(x), \fsl{D}^{\prime} w^{\prime}(x)\right) .
\end{aligned}
$$

In particular when the dimension $n=1$, using the realization of $\mathbb S$ as a left ideal in the exterior algebra,  one obtains
$$
B_{0 ; \lambda, \mu}^{(1)}=2 \mu(2 \mu+1) \frac{\partial^2}{\partial x^2}+2 \lambda(2 \lambda+1) \frac{\partial^2}{\partial y^2}-2(2 \lambda+1)(2 \mu+1) \frac{\partial^2}{\partial x \partial y}.
$$
This coincides (up to a constant multiple) to the degree two Rankin-Cohen operator \eqref{RC-intro} for the group  $\mathrm{SL}(2, \mathbb{R})$ which is isomorphic to $\operatorname{Spin}_0(1,2)$. \\ 

\noindent {\bf Other examples.}  Although the focus in this article is on the tensor product cases, the source operator method has been applied succesfully in other geometric contexts. Let us briefly mention these applications. \emph {Juhl operators} (see \cite{j}) are conformal symmetry breaking differential operators corresponding to the restriction from the sphere $S^n$ to $S^{n-1}$, more precisely from the conformal group ${\rm O}(1,n+1)$ to ${\rm O}(1,n)$ acting on densities. A construction of these operators has been presented by the second author in \cite{c17}. In the same context, a similar construction has been worked out for the case of differential forms and for the case of spinors (\cite{fo1}). There is a similar approach in conformal analysis of curved spades, and an alternative construction of the GJMS operators was proposed by B. \O rsted and  A. Juhl along similar lines (see \cite {jo}). It is likely that these ideas can be used in other Cartan geometries.



\end{document}